\documentclass[11pt,reqno,a4paper]{amsart}

\usepackage{microtype, textcomp,fbb}
\usepackage{
    amsmath,  amssymb,  amsthm,   amscd,
    gensymb,  graphicx, etoolbox, 
    booktabs, stackrel, mathtools  
}
\usepackage{ragged2e}
\usepackage{float}
\usepackage[english]{babel}
\usepackage{csquotes}
\usepackage{enumerate}
\usepackage{geometry}
\usepackage{multicol}
\usepackage{tasks}
\usepackage{lipsum}
\usepackage{fullpage}
\usepackage{wrapfig}
\usepackage{animate}
\usepackage{caption}
\usepackage{subcaption}
\usepackage{setspace}
\usepackage{tikz}
\usetikzlibrary{arrows.meta}
\usepackage{breqn}
\usepackage{csquotes}
\usepackage{mwe}
\definecolor{darkblue}{rgb}{0,0,0.7}
\definecolor{darkred}{rgb}{0.7,0,0}
\usepackage[colorlinks=true, linkcolor=darkred, citecolor=darkblue, urlcolor=blue, breaklinks=true]{hyperref}
\usepackage{cleveref}

\usepackage{chngcntr}
\usepackage{apptools}

\usepackage{enumitem}
\setlist{  
  listparindent=\parindent,
  parsep=0pt,
}

\newtheorem{theorem}{Theorem}[section]

\newtheorem{conjecture}{Conjecture}[section]
\theoremstyle{definition}
\newtheorem{definition}{Definition}[section]
\newtheorem{example}{Example}[definition]

\theoremstyle{theorem}
\newtheorem{prop}[theorem]{Proposition}
\newtheorem{lemma}[theorem]{Lemma}
\newtheorem{corollary}[theorem]{Corollary}
\newtheorem{claim}{Claim}[theorem]

\newcommand{\totalkcut}{\Delta_{k}^t}

\newcommand{\totalthreecut}{\Delta_{3}^t}
\newcommand{\totalkcutn}{\Delta_{k-1}^t}

\newcommand{\kcut}{\Delta_{k}}

\newcommand{\twogg}{\mathcal{G}_{2\times n}}
\newcommand{\twoggn}{\mathcal{G}_{2\times \left(n+1\right)}}
\newcommand{\twoggp}{\mathcal{G}_{2\times n}'}
\newcommand{\mgg}{\mathcal{G}_{m\times n}}

\newcommand{\threegg}{\mathcal{G}_{3\times n}}
\newcommand{\mthreegg}{\mathcal{G}_{3\times m}}

\newcommand{\mthreeggn}{\mathcal{G}_{3\times \lno m+1 \rno}}

\newcommand{\lno}{\left(}
\newcommand{\rno}{\right)}
\newcommand{\lsq}{\left[}
\newcommand{\rsq}{\right]}
\newcommand{\lcu}{\left\{}
\newcommand{\rcu}{\right\}}

\newcommand{\bbS}{\mathbb{S}}

\newcommand{\K}{\mathcal{K}}

\newcommand{\lk}{\mathrm{lk}}
\newcommand{\del}{\mathrm{del}}
\newcommand{\st}{\mathrm{st}}
\newcommand{\cone}{\mathrm{Cone}}

\definecolor{wwwwww}{rgb}{0.4,0.4,0.4}

\theoremstyle{remark}

\usepackage{pgfplots}
\pgfplotsset{compat=1.15}
\usepackage{mathrsfs}
\usetikzlibrary{arrows}

\makeatother

\begin{document}

\title{Topology of total cut complexes and cut complexes of grid graphs}

\author{Himanshu Chandrakar}
\author{Nisith Ranjan Hazra}
\author{Debotosh Rout}
\author{Anurag Singh}

\newcommand{\Addresses}{
  \bigskip
  \footnotesize

  \textsc{Department of Mathematics, Indian Institute of Technology Bhilai, India}\par\nopagebreak
  \textit{E-mail address}: HC: \texttt{himanshuc@iitbhilai.ac.in} \\
  \hspace*{2.45cm}NRH: \texttt{nisith.mds2024@cmi.ac.in} \\
  \hspace*{2.45cm}DR: \texttt{dr.24p10068@nitdgp.ac.in}\\
  \hspace*{2.45cm} AS: \texttt{anurags@iitbhilai.ac.in}
}

%\date{}

\begin{abstract}
    Inspired by the work of Fr{\"o}berg (1990) and Eagon and Reiner (1998), Bayer et al. recently introduced two new graph complexes: \emph{total cut complexes} and \emph{cut complexes}. In this article, we investigate these complexes specifically for (rectangular) grid graphs, focusing on $2 \times n$ and $3 \times n$ cases. We extend and refine the work of Bayer et al., proving and strengthening several of their conjectures, thereby enhancing the understanding of these graph complexes' topological and combinatorial properties.

\end{abstract}

\maketitle

\section{Introduction}

A \emph{graph complex} is a simplicial complex constructed from a graph $G$ based on certain combinatorial properties of $G$. One of the most influential examples of a graph complex is the \emph{neighborhood complex}, introduced by Lov{\'a}sz in 1978 \cite{lovasz1978}. This complex was instrumental in solving the Kneser conjecture, where Lov{\' a}sz provided a lower bound for the chromatic number of a graph in terms of the topological connectivity of its neighborhood complex. Since then, numerous graph complexes have been defined and studied (see, for instance, \cite{Henryvietoris, BK06, Barmakstarclusters,  Singh20, Singhhighermatching, Wachschessboard}), revealing various connections between their topological properties and the combinatorial properties of graphs (to mention a few, see \cite{Alonchromatic, Bjornercombinatorial, deshpandedomination, DS12, matsushita2023dominance, Meshulamdomination}). Additionally, graph complexes have found applications in many other areas of mathematics (see, for instance, \cite{Aharoniindependent, Bright04, camara2016topological, chanviralevolution, Dochtermannwarmath}). Beyond these connections, graph complexes have emerged as significant objects of study in their own right, highlighting their independent mathematical interest. For more comprehensive insights into the intrinsic topology of graph complexes, readers are referred to the books of Jonsson \cite{jonsson2008simplicial} and Kozlov \cite{kozlovCombAlgTopo}.

One notable intersection of graph complexes with other mathematical domains is in commutative algebra (see, for instance, \cite{Dochter09, dochter12, hibi92, hoch16, Stan96, van13}). A seminal connection was made by Fr{\"o}berg through his celebrated theorem \cite[Theorem 1]{froberg1990}, which bridges commutative algebra and graph theory via topological methods. 
%Eagon and Reiner further explored this intersection \cite[Proposition 8]{ER98}, providing a key reformulation of Fr{\"o}berg's theorem. Motivated by these foundational works, Bayer et al. recently introduced two novel graph complexes.
We briefly outline this connection before proceeding further. 

For a graph $G$, the \emph{clique complex}, say $\Delta(G)$, is the simplicial complex consisting of all subsets of vertices that form cliques in $G$; i.e., every two vertices in such a subset are adjacent in $G$. Fr{\"o}berg’s theorem concerns monomial ideals generated by quadratic terms and establishes that the Stanley–Reisner ideal $I_\Delta$ associated with a simplicial complex $\Delta$ has a $2$-linear resolution if and only if $\Delta$ is the clique complex $\Delta(G)$ of a chordal graph $G$. 
Eagon and Reiner later reformulated \cite[Proposition 8]{ER98} this result using the Alexander dual of the clique complex $\Delta(G)$, denoted $\Delta_2(G)$, whose facets correspond to the complements of independent sets of size 2. 

One can observe that the Fr{\"o}berg's theorem concerns ideals generated by quadratic monomials. A natural question to ask is about the ideals generated by higher-degree monomials. Motivated by this, Bayer et al. recently explored the following two generalizations of $\Delta_2(G)$ that arise from ideals generated by higher-degree monomials.

\begin{enumerate}
    \item Viewing $\Delta_2(G)$ as the simplicial complex whose facets are the complements of independent sets of size two in $G$, we obtain a natural generalization by taking complements of independent sets of size $k$ for any $k \ge 2$. This yields the \emph{total $k$-cut complex}, denoted $\Delta_k^t(G)$ (see \cref{def:totalcutcomplex}).

    \item Alternatively, if we interpret $\Delta_2(G)$ as the complex whose facets correspond to complements of two-element subsets of vertices whose induced subgraphs are disconnected, then we can generalize by considering complements of subsets of size $k \ge 2$ whose induced subgraph is disconnected in $G$. This gives the \emph{$k$-cut complex}, denoted $\Delta_k(G)$ (see \cref{def:cutcomplex}).
\end{enumerate}

Bayer et al. \cite{bayer2024cut, bayer2024total, bayer2024cutB} investigated these two complexes in detail, focusing on topological properties analogous to those appearing in Fr{\"o}berg’s theorem. They examined the total cut and 
$k$-cut complexes across various families of graphs, including complete multipartite graphs, cycles, squared cycles, prisms over cliques, trees, threshold graphs, triangle-free graphs, and grid graphs. More recently, Chouhan et al. \cite{Chauhan2025} studied the 3-cut complexes of squared cycle graphs and analyzed their shellability.

In this article, we focus on the total cut complexes and cut complexes associated with $m \times n$ grid graphs, denoted $\mgg$. Graph complexes arising from grid graphs have attracted considerable attention in topological combinatorics because of the inherent difficulty in determining their topological properties (see, e.g., \cite{BH17, chandrakar2024perfect,Matsushitamatching,Matsushitantimes4}). Bayer et al. showed that the complex $\Delta_2^t(\mgg) = \Delta_2(\mgg)$ is homotopy equivalent to a wedge of equi-dimensional spheres. Based on computational experiments, they further proposed the following conjecture.

\begin{conjecture}[\cite{bayer2024total}, Conjecture 5.2] 
 The nonzero Betti numbers satisfy the following formulas:
        \begin{equation*}
            \begin{split}
                b_{2n-2k}(\Delta_k^t(\twogg))& =\binom{n-1}{k-1}, ~k\geq 2,\\
                b_{3n-6}(\Delta_3^t(\threegg))& =\binom{2n-2}{2}.
            \end{split}
        \end{equation*}
\end{conjecture}

As our first main result, we confirm the above conjecture and further strengthen it by establishing the following theorem.

\begin{theorem}[Theorems \ref{thm: homotopy type of total k-cut complex of 2n-gg} and \ref{thm: homotopy of 3-cut complex of (3×n)-gg}]

We have the following homotopy equivalences, 
\begin{equation*}
            \begin{split}
                \totalkcut\lno \twogg\rno& \simeq \bigvee\limits_{\binom{n-1}{k-1}} \bbS^{\lno 2n-2k\rno}, ~ \text { for any } n\geq k\geq 2, \text{ and }\\
                \totalthreecut \lno \threegg \rno& \simeq\bigvee\limits_{\binom{2n-2}{2}}\bbS^{3n-6}, ~ \text { for any } n\geq 2.
            \end{split}
\end{equation*}
\end{theorem}

In a subsequent work \cite{bayer2024cutB}, Bayer et al. further examined the complexes $\Delta_k(\mgg)$ for $k = 3$ and $k = 4$, determining their homotopy types and proving that $\Delta_3(\mgg)$ is shellable. They also proposed the following conjecture.

\begin{conjecture}[\cite{bayer2024cutB}, Conjecture 5.9]  $\Delta_k(\mgg)$ is shellable for all $3 \leq k \leq mn -3$.
\end{conjecture}

We verify this conjecture in the case $m=2$. More precisely, we establish the following result.

\begin{theorem}[\Cref{lemma: Shellability of 2n'gg}]
   The complex $\kcut \lno \twogg \rno$ is shellable for all $ 3 \leq k \leq 2n-2 $. 
\end{theorem}

This article is organized as follows: In the next section, we discuss all the important definitions and preliminary results. \Cref{section: Total k-cut of 2n gg} discusses the homotopy type of the total $k$-cut complex of $\lno 2 \times n \rno$-grid graph, followed by \Cref{section:total cut 3*k} which discusses the homotopy type of the total $3$-cut complex of $\lno 3 \times n \rno$-grid graph. The main result proved in \Cref{section:total cut 3*k} uses tools from discrete Morse theory. In \Cref{section: shellability of k-cut of 2n-gg}, we discuss the shellability of the $k$-cut complex of $\lno 2 \times n \rno$-grid graph, where we use the recursive definition of shellability given by \cite[Definition~3.25]{jonsson2008simplicial} to prove the main result.

%We conclude the article by providing an \hyperref[appendix]{Appendix}, which discusses results used in \Cref{section:total cut 3*k} and whose proof is based on tools from discrete Morse theory.

\section{Preliminaries}
\subsection{Graph theoretic notations}
A (simple) \textit{graph}, say $G$, is an ordered pair $\lno V\lno G\rno, E\lno G \rno \rno$, where $V\lno G\rno$ is called the set of vertices and $E\lno G\rno \subseteq V\lno G\rno \times V\lno G\rno$ is called the set of edges. In this article, all the graphs are undirected and simple, i.e., $(u,v)=(v,u)$ and $(u,u)\notin E(G)$ for all $u,v\in V(G)$. For $S\subseteq V\lno G \rno$, the \textit{induced subgraph} of $G$ on the set $S$, denoted $G\lsq S \rsq$, is a graph with vertex set $S$ and edge set containing all those edges of $G$ whose endpoints are in $S$.

In this article, the graphs of our interest are $\lno 2 \times n \rno$ and $\lno 3 \times n\rno$ grid graphs, where $n \geq 1$. For $n\geq 1$, the $\lno 2 \times n \rno$-grid graph, denoted $\twogg$ is defined as follows:
\begin{align*}
    V\lno \twogg\rno = &\lcu a_i, b_i\ \middle|\ 1\leq i \leq n\rcu;\\
    E\lno \twogg\rno = &\lcu \lno a_i, a_{i+1}\rno, \lno b_i, b_{i+1}\rno\in V\lno \twogg\rno\times V\lno \twogg\rno\ \middle|\ 1\leq i \leq n-1\rcu\\& \sqcup \lcu \lno a_j, b_{j}\rno\in V\lno \twogg\rno\times V\lno \twogg\rno\ \middle|\ 1\leq j \leq n\rcu.
\end{align*}

We will define the $\lno 3 \times n \rno$-grid graph in \Cref{section:total cut 3*k}. However, let us define the graph $\twoggp$, for $n\geq 1$, as $\twoggp = \twoggn \lsq V(\twoggn)\setminus \lcu b_{n+1} \rcu \rsq$ (see \Cref{fig: graph of G'(2×n)}). This graph will be useful in \Cref{section: Total k-cut of 2n gg} and \Cref{section: shellability of k-cut of 2n-gg}.

\begin{figure}[h!]
\vspace{-1em}
\centering
    \begin{minipage}{0.5\textwidth}
        \hspace{2.5em}\begin{tikzpicture}[line cap=round,line join=round,>=triangle 45,x=0.4cm,y=0.4cm]
        \clip(-1,-2) rectangle (14,4);
        \draw [line width=1pt] (0,0)-- (0,3);
        \draw [line width=1pt] (3,0)-- (3,3);
        \draw [line width=1pt] (6,0)-- (6,3);
        \draw [line width=1pt] (6,3)-- (7,3);
        \draw [line width=1pt] (6,0)-- (7,0);
        \draw [line width=1pt] (0,3)-- (6,3);
        \draw [line width=1pt] (0,0)-- (6,0);
        \draw [line width=1pt,dash pattern=on 4pt off 4pt] (7,3)-- (9,3);
        \draw [line width=1pt,dash pattern=on 4pt off 4pt] (7,0)-- (9,0);
        \draw [line width=1pt] (9,3)-- (10,3);
        \draw [line width=1pt] (9,0)-- (10,0);
        \draw [line width=1pt] (10,0)-- (10,3);
        \draw [line width=1pt] (13,0)-- (13,3);
        \draw (-0.5,4.2) node[anchor=north west] {$a_{_1}$};
        \draw (12.5,4.2) node[anchor=north west] {$a_{_{n}}$};
        \draw (9.5,4.2) node[anchor=north west] {$a_{_{n-1}}$};
        \draw (5.5,4.2) node[anchor=north west] {$a_{_3}$};
        \draw (2.5,4.2) node[anchor=north west] {$a_{_2}$};
        \draw (-0.5,0) node[anchor=north west] {$b_{_1}$};
        \draw (12.5,0) node[anchor=north west] {$b_{_{n}}$};
        \draw (9.5,0) node[anchor=north west] {$b_{_{n-1}}$};
        \draw (5.5,0) node[anchor=north west] {$b_{_3}$};
        \draw (2.5,0) node[anchor=north west] {$b_{_2}$};
        \draw [line width=1pt] (10,3)-- (13,3);
        \draw [line width=1pt] (10,0)-- (13,0);
        \begin{scriptsize}
        \draw [fill=wwwwww] (0,0) circle (2pt);
        \draw [fill=wwwwww] (0,3) circle (2pt);
        \draw [fill=wwwwww] (3,0) circle (2pt);
        \draw [fill=wwwwww] (3,3) circle (2pt);
        \draw [fill=wwwwww] (6,0) circle (2pt);
        \draw [fill=wwwwww] (6,3) circle (2pt);
        \draw [fill=wwwwww] (10,3) circle (2pt);
        \draw [fill=wwwwww] (10,0) circle (2pt);
        \draw [fill=wwwwww] (13,0) circle (2pt);
        \draw [fill=wwwwww] (13,3) circle (2pt);
        \end{scriptsize}
        \end{tikzpicture}
        \vspace{-1em}
        \caption{The graph $\twogg$.}
        \label{fig: graph of G(2×n)}
    \end{minipage}%
    ~\hfill
    \begin{minipage}{0.5\textwidth}
        \hspace{1em}\begin{tikzpicture}[line cap=round,line join=round,>=triangle 45,x=0.4cm,y=0.4cm]
        \clip(-1,-2) rectangle (17.5,4);
        \draw [line width=1pt] (0,0)-- (0,3);
        \draw [line width=1pt] (3,0)-- (3,3);
        \draw [line width=1pt] (6,0)-- (6,3);
        \draw [line width=1pt] (6,3)-- (7,3);
        \draw [line width=1pt] (6,0)-- (7,0);
        \draw [line width=1pt] (0,3)-- (6,3);
        \draw [line width=1pt] (0,0)-- (6,0);
        \draw [line width=1pt,dash pattern=on 4pt off 4pt] (7,3)-- (9,3);
        \draw [line width=1pt,dash pattern=on 4pt off 4pt] (7,0)-- (9,0);
        \draw [line width=1pt] (9,3)-- (10,3);
        \draw [line width=1pt] (9,0)-- (10,0);
        \draw [line width=1pt] (10,0)-- (10,3);
        \draw [line width=1pt] (13,0)-- (13,3);
        \draw [line width=1pt] (13,3)-- (16,3);
        \draw (-0.5,4.2) node[anchor=north west] {$a_{_1}$};
        \draw (12.5,4.2) node[anchor=north west] {$a_{_{n}}$};
        \draw (15.5,4.2) node[anchor=north west] {$a_{_{n+1}}$};
        \draw (9.5,4.2) node[anchor=north west] {$a_{_{n-1}}$};
        \draw (5.5,4.2) node[anchor=north west] {$a_{_3}$};
        \draw (2.5,4.2) node[anchor=north west] {$a_{_2}$};
        \draw (-0.5,0) node[anchor=north west] {$b_{_1}$};
        \draw (12.5,0) node[anchor=north west] {$b_{_{n}}$};
        \draw (9.5,0) node[anchor=north west] {$b_{_{n-1}}$};
        \draw (5.5,0) node[anchor=north west] {$b_{_3}$};
        \draw (2.5,0) node[anchor=north west] {$b_{_2}$};
        \draw [line width=1pt] (10,3)-- (13,3);
        \draw [line width=1pt] (10,0)-- (13,0);
        \begin{scriptsize}
        \draw [fill=wwwwww] (0,0) circle (2pt);
        \draw [fill=wwwwww] (0,3) circle (2pt);
        \draw [fill=wwwwww] (3,0) circle (2pt);
        \draw [fill=wwwwww] (3,3) circle (2pt);
        \draw [fill=wwwwww] (6,0) circle (2pt);
        \draw [fill=wwwwww] (6,3) circle (2pt);
        \draw [fill=wwwwww] (10,3) circle (2pt);
        \draw [fill=wwwwww] (10,0) circle (2pt);
        \draw [fill=wwwwww] (13,0) circle (2pt);
        \draw [fill=wwwwww] (13,3) circle (2pt);
        \draw [fill=wwwwww] (16,3) circle (2pt);    
        \end{scriptsize}
        \end{tikzpicture}
        \vspace{-1em}
        \caption{The graph $\twoggp$.}
        \label{fig: graph of G'(2×n)}
    \end{minipage}
\end{figure}

\subsection{Topological results}\label{subsection: 2.2 Topological results} 

A basic familiarity with concepts such as the geometric realization of simplicial complexes and homotopy types is assumed; for details, see Hatcher’s Algebraic Topology~\cite{Hatcherat} and Kozlov’s Combinatorial Algebraic Topology~\cite{kozlovCombAlgTopo}.

An (abstract) \textit{simplicial complex} $\mathcal{K}$ is a collection of finite sets such that if $\sigma \in \mathcal{K}$ and $\tau \subseteq \sigma$, then $\tau \in \mathcal{K}$. 
The elements of $\mathcal{K}$ are called \textit{faces} or \textit{simplices} of $\mathcal{K}$. If $\mathcal{K}$ has no face, then we say that $\mathcal{K}$ is the \textit{void complex}. If $\sigma$ is an element of $\mathcal{K}$ and $\sigma$ has $k+1$ elements, then $\sigma$ is said to be \textit{$k$-dimensional}. Thus, the empty simplex has dimension $\lno-1\rno$, and the \textit{vertices} of $\mathcal{K}$ are $0$-dimensional. The \emph{dimension} of a simplicial complex $\mathcal{K}$, denoted as dim$(\mathcal{K})$, is defined as the maximum dimension of its simplices. The simplices of $\mathcal{K}$ that are not properly contained in any other simplex of $\mathcal{K}$ are called \textit{facets}, and $\mathcal{K}$ is called \textit{pure} if all its facets have the same dimension.

A subset $ \mathcal{L} \subseteq \mathcal{K} $ is said to be a \textit{subcomplex} of a $\mathcal{K}$ if $\mathcal{L}$ itself is a simplicial complex. For any finite simplicial complex $\mathcal{K}$, the $d$-\textit{skeleton} of $\mathcal{K}$ is the collection of all simplices of $\mathcal{K}$ up to dimension $d$ and is denoted by $\Delta^{(d)}$. 
  %The $\Delta^{dim\lno \mathcal{K} \rno}$ is called the \textit{boundary} of $\mathcal{K}$. 
For any set $F$, we denote by $\langle F \rangle$ a simplicial complex with only one facet $F$. 

There are various important subcomplexes associated with a given simplicial complex. We discuss some of them, which we will be using.
\begin{definition}
    Let $\mathcal{K}$ be a simplicial complex, and $\sigma$ be a face of $\mathcal{K}$. Then,
    \begin{enumerate}[label = \arabic*.]
        \item The \textit{link} of $\sigma$ in $\mathcal{K}$, denoted by $\lk_{\mathcal{K}}\lno \sigma \rno$,
        $$\lk_{\mathcal{K}}\lno \sigma \rno = \lcu \tau \in \mathcal{K}\ \middle|\ \sigma \cap \tau = \emptyset,\ \sigma \cup \tau \in \mathcal{K} \rcu.$$
        
        \item The \textit{deletion} of $\sigma$ in $\mathcal{K}$, denoted by $\del_{\mathcal{K}}\lno \sigma \rno$, is
        $$\del_{\mathcal{K}}\lno \sigma \rno = \lcu \tau \in \mathcal{K}\ \middle|\ \sigma \nsubseteq \tau \rcu.$$
        
        \item The (closed) \textit{star} of $\sigma$ in $\mathcal{K}$, denoted by $\st_{\mathcal{K}}\lno \sigma \rno$, is
        $$\st_{\mathcal{K}}\lno \sigma \rno = \lcu \tau \in \mathcal{K}\ \middle|\ \sigma \cup \tau \in \mathcal{K} \rcu.$$
    \end{enumerate}
\end{definition}

A vertex $v$ of $\mathcal{K}$ is called a \emph{shedding vertex} if every facet of $\del_{\mathcal{K}}\lno v\rno$ is a facet of $\mathcal{K}$.
For simplicial complexes $\mathcal{K}_1$ and $\mathcal{K}_2$, the \textit{join} of $\mathcal{K}_1$ and $\mathcal{K}_2$, denoted as $\mathcal{K}_1 * \mathcal{K}_2$, is a simplicial complex whose simplices are the disjoint union of simplices of $\mathcal{K}_1$ and $\mathcal{K}_2$. Because of the join operation, one can construct other complexes from given simplicial complexes. We will be using the following particular constructions:

\begin{definition}
Let $x, y \notin V\lno \mathcal{K} \rno$.
    \begin{enumerate}
        \item The \textit{cone} over $\mathcal{K}$ with apex $x$, denoted $\cone_{x} \lno \mathcal{K} \rno$, is defined as
        $$\cone_{x} \lno \mathcal{K} \rno = \mathcal{K} * \lcu x \rcu.$$
    
        \item  The \textit{suspension} of $\mathcal{K}$, denoted $\Sigma \lno \mathcal{K} \rno$, is defined as
        $$\Sigma \lno \mathcal{K} \rno = \lno \mathcal{K} * \lcu x \rcu \rno \cup_{\mathcal{K}} \lno \mathcal{K} * \lcu y \rcu \rno .$$
    \end{enumerate}
\end{definition}

We have all the required definitions to discuss a few important results that will be used extensively in \Cref{section: Total k-cut of 2n gg} and \Cref{section:total cut 3*k}.

\begin{lemma}[\cite{Matsushitantimes4}, Lemma~2.1]\label{lemma: finding homotopy using link and deletion}
    Let $\mathcal{K}$ be a simplicial complex and $v$ be a vertex of $\mathcal{K}$. If $\lk_{\mathcal{K}}\lno v\rno$ is contractible in $\del_{\mathcal{K}}\lno v\rno$ then $$\mathcal{K} \simeq \del_{\mathcal{K}}\lno v\rno \vee \Sigma\lno \lk_{\mathcal{K}}\lno v\rno\rno.$$
\end{lemma}

Here, $\vee$ denotes the wedge of spaces, and the empty wedge would mean the space is contractible.

\begin{lemma}[\cite{BjornerTopoMethods}, Lemma~10.4(b)]\label{lemma: bjorner's lemma 10.4(b)}
    Let $\mathcal{K}$ be simplicial complex with subcomplexes $\mathcal{K}_i$, for $0\leq i\leq n$, such that, $\mathcal{K} = \bigcup\limits_{i=0}^n \mathcal{K}_i$. If $\mathcal{K}_i \cap \mathcal{K}_j \subseteq \mathcal{K}_0$, for all $1\leq i < j \leq n$ and $\mathcal{K}_i$ is contractible for all $0\leq i \leq n$, then $$\mathcal{K} = \bigvee\limits_{1\leq i \leq n}\Sigma\lno\mathcal{K}_0 \cap \mathcal{K}_i \rno.$$
\end{lemma}

Another useful combinatorial tool that reveals important topological information about a simplicial complex is shellability.

\begin{definition}[\cite{kozlovCombAlgTopo}, Defintion~12.1]\label{def:shellable}
    A simplicial complex $\mathcal{K}$ is said to be shellable if the facets of $\mathcal{K}$ can be put together in a linear order $F_1, F_2, \dots, F_t$ such that the subcomplex $\lno \bigcup_{s=1}^{j-1}\langle F_s \rangle \rno \cap \langle F_j \rangle$ is pure and of dimension $\lno dim \langle F_j \rangle -1 \rno$ for $2 \leq j \leq t$. Such an ordering on $\mathcal{K}$ is called a shelling order.
\end{definition}   
The definition of shellability mentioned above is commonly used. An alternative definition, provided by Bj{\"o}rner and Wachs (\cite[Lemma~2.3]{Bjornershellable}), states that a linear order $\lno F_1, F_2, \ldots, F_t \rno$ of the facets of $\mathcal{K}$ is a shelling order if and only if for all $ 1 \leq i < j \leq t $, there exists $k < j$  such that $F_i \cap F_j \subseteq F_k \cap F_j$ and $|F_k \cap F_j| = |F_j| - 1$.

Now, for our discussion, we will be using the following recursive definition of shellability given by Jonsson:

\begin{definition}[\cite{jonsson2008simplicial}, Definition~3.25]\label{def: shellability definition we used}
    The class of shellable simplicial complexes is recursively defined as follows:
    \begin{enumerate}
        \item Every simplex is shellable.
        \item If $\mathcal{K}$ is pure and contains a shedding vertex $v$ such that $\lk_{\mathcal{K}}\lno v\rno$ and $\del_{\mathcal{K}}\lno v\rno$ are shellable then $\mathcal{K}$ is shellable.
    \end{enumerate}
\end{definition}

\begin{lemma}[\cite{Bjornershellable}, Theorem~2.9]\label{shellability of skeleton}
	If $\mathcal{K}$ is shellable, then its $d$-skeleton $\mathcal{K}^{(d)}$ is also shellable, for all $d\geq 0$.
\end{lemma}

Let us now mention the following result, which states that the join of two shellable simplicial complexes is shellable.

\begin{lemma}[\cite{Wachs97b}, Remark 10.22]\label{join is shellable}
    The join of two simplicial complexes is shellable if and only if each is shellable.
\end{lemma}

\subsection{Discrete Morse Theory}
Discrete Morse theory is one of the most important tools in combinatorial topology, formulated by Robin Forman (see \cite{forman1998morse, forman2002user}) from the classical Morse theory to compute the homotopy type of various polyhedral complexes, particularly simplicial complexes. In this section, we will be following the notations and definitions mainly from the book of Kozlov \cite{kozlovCombAlgTopo}.

\begin{definition} (\cite[{Definition~11.1}]{kozlovCombAlgTopo}) Let $(\mathcal{P}, <)$ be a partially ordered set (poset).
    \begin{enumerate}
        \item\label{defn: Partial matching} A \emph{partial matching} in $\mathcal{P}$ is a subset $\mathcal{M} \subseteq \mathcal{P}\times\mathcal{P}$, such that
        \begin{itemize}
            \item $\left(a,b\right) \in \mathcal{M}$ implies $a \prec b$;
            \item each $c \in \mathcal{P}$ belongs to at most one element (pair) of $\mathcal{M}$.
        \end{itemize}
        Here, $a \prec b$ means there exists no $\delta \in \mathcal{P}$ such that $a < \delta < b$. Moreover, note that $\mathcal{M}$ is a partial matching on poset $\mathcal{P}$ if and only if there exists $\mathcal{T} \subset \mathcal{P}$ and an injective map $\psi: \mathcal{T} \to \mathcal{P} \setminus \mathcal{T}$ such that $t \prec \psi\left(t\right)$ for all $t \in \mathcal{T}$.

        \item\label{Acyclic matching} A partial matching on $\mathcal{P}$ is said to be \emph{acyclic} if there does not exist a cycle, $$a_1 \prec \psi\left(a_1\right) \succ a_2 \prec \psi\left(a_2\right) \succ \dots \succ a_k \prec \psi\left(a_k\right) \succ a_1,$$
        where $k\geq2$ and all $a_i \in \mathcal{P}$ being distinct.
    \end{enumerate}

    For an acyclic matching $\mathcal{M}$ on poset $\mathcal{P}$, we define \textit{critical elements} to be those elements of $\mathcal{P}$ which remain unmatched, i.e., the elements of $\mathcal{P}\setminus (\mathcal{T}\cup \psi(\mathcal{T}))$.
\end{definition}

 We now state the main theorem of discrete Morse theory.

\begin{theorem}\label{theorem 1}(\cite[{Theorem~11.13}]{kozlovCombAlgTopo})
    Let $\mathcal{K}$ be a simplicial complex, and let $\mathcal{M}$ be an acyclic matching on the face poset of $\mathcal{K}$. If $c_i$ denotes the number of critical $i$-dimensional cells of $\mathcal{K}$ then the space $\mathcal{K}$ is homotopy equivalent to a cell complex $\mathcal{K}_c$ with $c_i$ cells of dimension $i$ for each $i \geq 0$, plus a single 0-dimensional cell in the case where the empty set is paired in the matching.
\end{theorem}

Based on the above theorem, the following information can be obtained:

\begin{corollary}\label{Corollary 1}(\cite[Corollary~2.5]{deshpande2020higher})
    For an acyclic matching $\mathcal{M}$, if all the critical cells in $\mathcal{M}$ are of dimension $d$, then $\mathcal{K}$ is homotopy equivalent to a wedge of $d$-dimensional spheres.
\end{corollary}

An easier way to pair elements of the face poset is to perform \textit{element matching} on them. A proper definition of element matching is as follows:

\begin{definition}[\cite{jonsson2008simplicial}] \label{def: Element matching}
    Let $v$ be a vertex and $\Delta$ be a simplicial complex. The \emph{element matching using $v$ on $\Delta$} is defined as the following set of pairs, \[M_{v} = \left\{\left(\sigma,\sigma \cup \left\{v\right\}\right)\ |\ v \notin \sigma,\ \sigma \cup \left\{v\right\} \in \Delta\right\}.\]
\end{definition}

The advantage of defining a sequence of element matching is that the union of a sequence of element matching is an acyclic matching, as stated in the following theorem.

\begin{prop}\label{theorem 2}(\cite[Proposition~3.1]{deshpande2020higher}, \cite[Lemma~4.1]{jonsson2008simplicial})
    Let $\Delta$ be a simplicial complex and $\{v_{1},v_{2},\dots, $ $v_{t}\}$ be a subset of the vertex set of $\Delta$. Let $\Delta_0 = \Delta$, and for all $i \in \left\{1,\ 2,\dots,\ t\right\}$, define
    \begin{align*}
        M_{v_{i}} &= \left\{\left(\sigma,\ \sigma \cup \left\{v_{i}\right\}\right)\ \middle|\ v_{i} \notin \sigma, \text{and}\ \sigma,\ \sigma \cup \left\{v_{_i}\right\} \in \Delta_{i-1}\right\},\\
        N_{v_{i}} &= \left\{\sigma \in \Delta_{i-1}\ \middle|\ \sigma \in \eta\ \text{for some}\ \eta \in M_{v_{i}}\right\},\ \text{and}\\
        \Delta_i &= \Delta_{i-1} \setminus N_{v_{i}}.
    \end{align*}
    Then, $\bigsqcup_{i=1}^t{M_{v_{_i}}}$ is an acyclic matching on $\Delta$.
\end{prop}

\subsection{Total cut complexes and cut complexes}
We now define the total cut and cut complexes associated with graphs.

\begin{definition}[\cite{bayer2024total}, Definition~2.3]\label{def:totalcutcomplex}
Let $G$ be a graph and $k \ge 1$. The \emph{total $k$-cut complex} of $G$, denoted by $\totalkcut\lno G\rno$, is the simplicial complex whose facets are the complements of independent sets of size $k$. Equivalently, $\sigma \in \totalkcut\lno G\rno$ if and only if $V(G) \setminus \sigma$ contains an independent set of size $k$.
\end{definition}

\begin{example}
    Let $G_1$ be the graph shown in \Cref{fig: the graph G1}.
    The collection of independent set of size 3 in $G_1$, say $I_3 \lno G_1 \rno$ is $$I_3 \lno G_1 \rno = \lcu \lno 135 \rno, \lno 136 \rno, \lno 145 \rno, \lno 235 \rno \rcu.$$
    
    Here, $\lno abc \rno = \lcu a, b,c\rcu$, for some $a,b,c \in V_1$. Consequently, the total-3 cut complex of $G_1$ (see \Cref{fig: total 3-cut of G1}) has complements of the elements of $I_3 \lno G_1 \rno$ as its facets, i.e., $$\totalthreecut \lno G_1 \rno = \left\langle \lno 146 \rno, \lno 236 \rno, \lno 245 \rno, \lno 246 \rno \right\rangle.$$
\end{example}

As another example, we consider the total $3$-cut complex of a grid graph. 
\begin{figure}[H]
    \hspace{-2em}
    \begin{subfigure}[b]{0.5\textwidth}
    \centering
        \begin{tikzpicture}[line cap=round,line join=round,>=triangle 45,x=0.35cm,y=0.35cm, roundnode/.style={scale = 0.5, circle, draw=cyan!60, fill=cyan!5, thick, minimum size=2.5em}]
        \clip(1,1) rectangle (11,11);
        \draw [line width=1pt] (4,4)-- (8,4);
        \draw [line width=1pt] (8,4)-- (6,7.464101615137754);
        \draw [line width=1pt] (6,7.464101615137754)-- (4,4);
        \draw [line width=1pt] (4,4)-- (2,2);
        \draw [line width=1pt] (8,4)-- (10,2);
        \draw [line width=1pt] (6,7.464101615137754)-- (6,10.29);
        
        \draw (6,10.29) node [roundnode] {\huge \textbf{1}};
        \draw (6,7.464101615137754) node [roundnode] {\huge \textbf{2}};
        \draw (2,2) node [roundnode] {\huge \textbf{3}};
        \draw (4,4) node [roundnode] {\huge \textbf{4}};
        \draw (10,2) node [roundnode] {\huge \textbf{5}};
        \draw (8,4) node [roundnode] {\huge \textbf{6}};
    \end{tikzpicture}
    \caption{The graph $G_1$.}
    \label{fig: the graph G1}
    \end{subfigure}%
    \hspace{-3em}
    \begin{subfigure}[b]{0.5\textwidth}
    \centering
    \hspace{-1em}
        \begin{tikzpicture}[line cap=round,line join=round,>=triangle 45,x=0.55cm,y=0.55cm, roundnode/.style={scale = 0.5, circle, draw=cyan!60, fill=cyan!5, thick, minimum size=2.5em}]
        \clip(0,0) rectangle (5.5,6.5);
        \fill[black!15] (1,1) -- (1,4) -- (3,2.5) -- cycle;
        \fill[black!15] (1,4) -- (3,2.5) -- (5,4) -- cycle;
        \fill[black!15] (5,4) -- (3,2.5) -- (5,1) -- cycle;
        \fill[black!15] (1,4) -- (5,4) -- (3,6) -- cycle;
        
        \draw [line width=1pt] (1,1)-- (1,4);
        \draw [line width=1pt] (1,4)-- (3,2.5);
        \draw [line width=1pt] (3,2.5)-- (1,1);
        \draw [line width=1pt] (3,2.5)-- (5,4);
        \draw [line width=1pt] (5,4)-- (5,1);
        \draw [line width=1pt] (5,1)-- (3,2.5);
        \draw [line width=1pt] (1,4)-- (5,4);
        \draw [line width=1pt] (5,4)-- (3,6);
        \draw [line width=1pt] (3,6)-- (1,4);

        \draw (1,1) node [roundnode] {\huge \textbf{1}};
        \draw (5,4) node [roundnode] {\huge \textbf{2}};
        \draw (3,6) node [roundnode] {\huge \textbf{3}};
        \draw (3,2.5) node [roundnode] {\huge \textbf{4}};
        \draw (5,1) node [roundnode] {\huge \textbf{5}};
        \draw (1,4) node [roundnode] {\huge \textbf{6}};
    \end{tikzpicture}
    \caption{$\totalthreecut\lno G_1 \rno $.}
    \label{fig: total 3-cut of G1}
    \end{subfigure}
    \caption{Total $3-$cut complex of the graph $G_1$.}
\end{figure}

%\begin{example}
%    For $\mathcal{G}_{3\times3}$ (see \Cref{fig: graph of G(3×n)}), we have 
%    $$I_3 \lno \mathcal{G}_{3\times3} \rno = \lcu \begin{split}
%        &\lno a_1 a_3 b_2 \rno, \lno a_1 a_3 c_1 \rno, \lno a_1 a_3 c_2 \rno, \lno a_1 a_3 c_3 \rno, \lno a_1 b_2 c_1 \rno, \lno a_1 b_2 c_3 \rno,\\
%        &\lno a_1 b_3 c_1 \rno, \lno a_1 b_3 c_2 \rno, \lno a_1 c_1 c_3 \rno, \lno a_2 b_1 b_3 \rno, \lno a_2 b_1 c_2 \rno, \lno a_2 b_1 c_3 \rno,\\ &\lno a_2 b_3 c_1 \rno, \lno a_2 b_3 c_2 \rno, \lno a_2 c_1 c_3 \rno, \lno a_3 b_1 c_2 \rno, \lno a_3 b_1 c_3 \rno, \lno a_3 b_2 c_1 \rno,\\ &\lno a_3 b_2 c_3 \rno, \lno a_3 c_1 c_3 \rno, \lno b_1 b_3 c_2 \rno, \lno b_2 c_1 c_3 \rno
%    \end{split} \rcu.
%    $$
    
%    Then the total-3 cut complex of $\mathcal{G}_{3\times3}$ has facets which are complements of the elements of $I_3 \lno G_1 \rno$, i.e., $$\totalthreecut \lno G_{3\times3} \rno = \left\langle \sigma^c \in V\lno \mathcal{G}_{3\times3} \rno \ \middle|\ \sigma \in I_3 \lno \mathcal{G}_{3\times3} \rno \right\rangle.$$
%\end{example}

\begin{example}
    For $\mathcal{G}_{2\times4}$ (see \Cref{fig: graph of G(2×n)}), the collection of independent sets of size $3$ is
    $$I_3 \lno \mathcal{G}_{2\times4} \rno = \lcu \begin{split}
        &\lno a_1 a_3 b_2 \rno, \lno a_1 a_3 b_4 \rno, \lno a_1 a_4 b_2 \rno, \lno a_1 a_4 b_3 \rno, \lno a_1 b_2 b_4 \rno, \lno a_1 a_4 b_1 \rno,\\
        &\lno a_2 a_4 b_3 \rno, \lno a_2 b_1 b_3 \rno, \lno a_2 b_1 b_4 \rno, \lno a_3 b_1 b_4 \rno, \lno a_3 b_2 b_4 \rno, \lno a_4 b_1 b_3 \rno
    \end{split} \rcu.
    $$
    
    Then the total-3 cut complex of $\mathcal{G}_{2\times4}$ has facets which are complements of the elements of $I_3 \lno \mathcal{G}_{2\times4} \rno$, i.e., $$\totalthreecut \lno \mathcal{G}_{2\times4} \rno = \left\langle \begin{split}
        &\lno a_2 a_4 b_1 b_3 b_4 \rno,
        \lno a_2 a_4 b_1 b_2 b_3 \rno,
        \lno a_2 a_3 b_1 b_3 b_4 \rno,
        \lno a_2 a_3 b_1 b_2 b_4 \rno,\\
        &\lno a_2 a_3 a_4 b_1 b_3 \rno,
        \lno a_2 a_3 b_2 b_3 b_4 \rno,
        \lno a_1 a_3 b_1 b_2 b_4 \rno,
        \lno a_1 a_3 a_4 b_2 b_4 \rno,\\
        &\lno a_1 a_3 a_4 b_2 b_3 \rno,
        \lno a_1 a_2 a_4 b_2 b_3 \rno,
        \lno a_1 a_2 a_4 b_1 b_3 \rno,
        \lno a_1 a_2 a_3 b_2 b_4 \rno
    \end{split} \right\rangle.$$
\end{example}

\begin{definition}[\cite{bayer2024cut}, Definition~2.7]\label{def:cutcomplex}
Let $G = (V, E)$ be a graph and let $k \ge 1$. 
The \emph{$k$-cut complex} of $G$, denoted by $\kcut \lno G \rno$, 
is the simplicial complex whose facets are the complements of subsets of $V(G)$ of size $k$ whose induced subgraphs are disconnected. 
Equivalently, a subset $\sigma \subseteq V(G)$ is a face of $\kcut\lno G\rno$ if and only if its complement $V(G) \setminus \sigma$ induces a disconnected subgraph on $k$ vertices.
\end{definition}

\begin{example}
    Let $G_2$ be the graph shown in  \Cref{fig: graph G2}.    
    Let $S^{(3)}\lno G_2 \rno$ be the set of all 3-element subsets $\sigma$ of $V\lno \mathcal{G}_{2\times 3} \rno$ for which the induced subgraph $G_2 \lsq \sigma \rsq$ is disconnected, i.e., $$S^{(3)}\lno G_2 \rno = \lcu \lno 123 \rno, \lno 125 \rno, \lno 145 \rno, \lno 234 \rno, \lno 345 \rno\rcu.$$
    
    Consequently, the $3-$cut complex of $G_2$ (see \Cref{fig: 3-cut of the graph G2}) has complements of the elements of $S^{(3)} \lno G_2 \rno$ as its facets, i.e., $$\Delta_3 \lno G_2 \rno = \left\langle \lno 12 \rno, \lno 23 \rno, \lno 34 \rno, \lno 45 \rno, \lno 51 \rno \right\rangle.$$
\end{example}

\begin{figure}[H]
    \vspace{-1.5em}
    \hspace{-3em}
    \begin{subfigure}[b]{0.5\textwidth}
    \centering
    \hspace{1em}
    \begin{tikzpicture}[line cap=round,line join=round,>=triangle 45,x=0.75cm,y=0.75cm, roundnode/.style={scale = 0.5, circle, draw=green!60, fill=green!5, thick, minimum size=2.5em}]
        \clip(0.5,-0.5) rectangle (5,4);
        \draw [line width=1pt] (1.5,0)-- (4.118033988749895,1.9021130325903064);
        \draw [line width=1pt] (1.5,0)-- (2.5,3.077683537175253);
        
        \draw [line width=1pt] (3.5,0)-- (0.8819660112501053,1.9021130325903073);
        \draw [line width=1pt] (3.5,0)-- (2.5,3.077683537175253);
        
        \draw [line width=1pt] (0.8819660112501053,1.9021130325903073)-- (4.118033988749895,1.9021130325903064);
    
        \draw (2.5,3.077683537175253) node [roundnode] {\huge \textbf{1}};
        \draw (4.118033988749895,1.9021130325903064) node [roundnode] {\huge \textbf{2}};
        \draw (3.5,0) node [roundnode] {\huge \textbf{3}};
        \draw (1.5,0) node [roundnode] {\huge \textbf{4}};
        \draw (0.9563340376883458,1.9021130325903064) node [roundnode] {\huge \textbf{5}};
    \end{tikzpicture}
    \caption{The graph $G_2$.}
    \label{fig: graph G2}
    \end{subfigure}%
    \hspace{-3em}
    \begin{subfigure}[b]{0.5\textwidth}
    \centering
    \hspace{1em}
    \begin{tikzpicture}[line cap=round,line join=round,>=triangle 45,x=0.75cm,y=0.75cm, roundnode/.style={scale = 0.5, circle, draw=green!60, fill=green!5, thick, minimum size=2.5em}]
        \clip(0.5,-0.5) rectangle (5,4);
        \draw [line width=1pt] (1.5,0)-- (3.5,0);
        \draw [line width=1pt] (3.5,0)-- (4.118033988749895,1.9021130325903064);
        \draw [line width=1pt] (4.118033988749895,1.9021130325903064)-- (2.5,3.077683537175253);
        \draw [line width=1pt] (2.5,3.077683537175253)-- (0.8819660112501053,1.9021130325903073);
        \draw [line width=1pt] (0.8819660112501053,1.9021130325903073)-- (1.5,0);
    
        \draw (2.5,3.077683537175253) node [roundnode] {\huge \textbf{1}};
        \draw (4.118033988749895,1.9021130325903064) node [roundnode] {\huge \textbf{2}};
        \draw (3.5,0) node [roundnode] {\huge \textbf{3}};
        \draw (1.5,0) node [roundnode] {\huge \textbf{4}};
        \draw (0.9563340376883458,1.9021130325903064) node [roundnode] {\huge \textbf{5}};
    \end{tikzpicture}
    \caption{$\Delta_3\lno G_2 \rno$.}
    \label{fig: 3-cut of the graph G2}
    \end{subfigure}
    \caption{$3-$cut complex of the graph $G_2$.}
\end{figure}

\begin{example}
    For $\mathcal{G}_{2\times 3}$ (see \Cref{fig: graph of G(2×n)}), we have,
    $$ S^{(3)}\lno \mathcal{G}_{2\times 3} \rno = \lcu \begin{split}
        \lno a_1 a_2 b_3 \rno, \lno a_1 a_3 b_1 \rno, \lno a_1 a_3 b_2 \rno, \lno a_1 a_3 b_3 \rno, \lno a_1 b_1 b_3 \rno,\\
        \lno a_1 b_2 b_3 \rno, \lno a_2 a_3 b_1 \rno, \lno a_2 b_1 b_3 \rno, \lno a_3 b_1 b_2 \rno, \lno a_3 b_1 b_3 \rno
    \end{split} \rcu.$$

    Then the $3-$cut complex of $\mathcal{G}_{2\times 3}$ has facets which are complements of the elements of $S^{(3)} \lno G_2 \rno$, i.e., 
    $$ \Delta_3\lno \mathcal{G}_{2\times 3} \rno = \left\langle
    \begin{split}
        \lno a_1 a_2 b_2 \rno, \lno a_1 a_2 b_3 \rno, \lno a_1 a_3 b_2 \rno, \lno a_1 b_2 b_3 \rno, \lno a_2 a_3 b_1 \rno,\\
        \lno a_2 a_3 b_2 \rno, \lno a_2 b_1 b_2 \rno, \lno a_2 b_1 b_3 \rno, \lno a_2 b_2 b_3 \rno, \lno a_3 b_1 b_2 \rno
    \end{split} \right\rangle.$$
\end{example}

We now discuss some results related to $\totalkcut\lno G \rno$ and $\kcut \lno G \rno$ that will be used repeatedly throughout the article.

\begin{lemma}[\cite{bayer2024total}, Lemma~3.3]\label{lemma: link}
    Let $k\geq 2$, $G = \lno V,E\rno$ be a graph, and $W$ be a face of $\totalkcut\lno G\rno$. Then, 
    $$\totalkcut\lno G \setminus W\rno = \lk_{\totalkcut\lno G\rno}\lno W\rno.$$
\end{lemma}
\begin{corollary}\label{coro:link of a vertex}
    Let $G$ be a graph and $x\in V(G)$. Then, $lk_{\totalkcut\lno G\rno}\lno \lcu x\rcu\rno = \totalkcut\lno G\setminus \lcu x\rcu\rno$.
\end{corollary}

\begin{definition}
    A vertex $v$ of graph $G$ is called a \textit{simplicial vertex} if $G\lsq N\lno v\rno\rsq$ is a clique, i.e., every pair of distinct vertices in $N\lno v\rno$ is adjacent in $G$.
\end{definition}

\begin{lemma}[\cite{bayer2024total}, Lemma~3.5]\label{lemma: deletion of v}
    Let $k\geq2$ and $G = \lno V, E\rno$ be a graph with a simplicial vertex $v$. If $\Delta = \totalkcut\lno G\rno$ then $\del_{\Delta}\lno v\rno$ is a pure simplicial complex generated by the facets of $\totalkcutn\lno G\setminus\lcu v\rcu\rno$ that contain $N\lno v\rno$, i.e.,
    $$\del_{\Delta}\lno v\rno = \st_{\totalkcutn\lno G\setminus\lcu v\rcu\rno}\lno N\lno v\rno\rno.$$
\end{lemma}

A vertex $v\in V(G)$ is called a \emph{leaf} if $|N(v)|=1$; here $|\cdot |$ denotes the cardinality of a set. The following is a particular case of the above result.
\begin{corollary}\label{coro: deletion of a vertex}
    Let $G$ be a graph and $x$ be a leaf vertex of $G$. Then $\del_{\totalkcut\lno G\rno}\lno x\rno \simeq *.$
\end{corollary}

\begin{lemma}[\cite{bayer2024total}, Theorem~3.6] \label{thm: suspension of previous graph}
    Let $G$ be a graph and $v \in V\lno G\rno$ be a simplicial vertex with non-empty neighborhood $N\lno v\rno$ and $k\geq2$. Let $H = G \setminus \lcu v \rcu$. If $\totalkcut\lno H \rno$ is not the void complex then
    $$ \totalkcut \lno G \rno \simeq \Sigma \lno \totalkcut \lno G \setminus \lcu v \rcu \rno\rno,$$
    and hence there is an isomorphism in the reduced homology $\tilde{H}_{n} \lno \totalkcut \lno G \rno \rno \cong \tilde{H}_{n-1} \lno \totalkcut \lno H \rno \rno$, for all $n \geq 1$.
    If $\totalkcut\lno H \rno$ is the void complex, then $\totalkcut\lno G \rno$ is either contractible or the void complex.
\end{lemma}

The following result is similar to \Cref{lemma: link} for cut complexes.
\begin{lemma}[\cite{bayer2024cut}, Lemma 4.5]\label{lemma:link for cut}
    Let $k \geq 2$, $G$ be a graph, and $W \subseteq V(G)$. Then $\Delta_k(G \setminus W) = \lk_{\Delta_k(G)}(W)$ if $W$ is a face of $\Delta_k(G)$ and is void otherwise.
\end{lemma}

Finally, we end this section with the following identity, called the \emph{Hockey Stick Identity}, involving binomial coefficients. 

\begin{prop} [\cite{graham94}, Equation (5.10); \cite{hockey}, Equation (2); \cite{StanleyenumerativeI}, Exercise 3(a), p.104] \label{prop: binomial result}
    For $n\geq 1$ and $k\geq 0$, $$\sum\limits_{i=1}^k \binom{n-i}{k-i} = \binom{n}{k-1}.$$
\end{prop}

% \begin{proof}
%     We use Pascal's rule, $\binom{n}{t} + \binom{n}{t+1} = \binom{n+1}{t+1}$ to prove the result. The proof is by induction on $k$. For $k=0$, $$\binom{n-1}{1-1} = 1 = \binom{n}{1-1}.$$
%     Assuming the given result is true for $k-1$, we show that it is true for $k$, as follows    
%     \begin{align*}
%         \sum\limits_{i=1}^{k} \binom{n-i}{k-i} &= \binom{n-1}{k-1} + \sum\limits_{j=1}^{k-1} \binom{n-1-j}{k-1-j}\\
%         &= \binom{n-1}{k-1} + \binom{n-1}{k-2}\\
%         &= \binom{n}{k-1}.
%     \end{align*}
%     This completes the proof.
% \end{proof}

\section{Total $k$-cut complex  of $\lno 2\times n\rno$ Grid Graph}\label{section: Total k-cut of 2n gg}
In this section, we determine the homotopy type of $\totalkcut\lno \twogg\rno$ for all $n \ge 2$ and $1 \le k \le n$. 
The upper bound $k \le n$ arises because, in the grid graph $G_{2\times n}$, an independent set can have at most $n$ vertices. 
Consequently, for $k > n$, the total $k$-cut complex $\totalkcut\lno \twogg\rno$ is void.

It is easy to see that, for $k=1$ and any graph $G$, $\Delta_{k}^t (G)$ is precisely the boundary of a simplex on the vertex set $V(G)$. As a result, we have $$\Delta_{1}^t (G) \simeq \bbS^{|V(G)|-2}.$$

Hence, we can assume that $k\geq 2$. For simplicity of notations, let us denote $\totalkcut \lno \twogg \rno$ by $\Delta_{k,n}^t$. The following result will be useful in finding the homotopy of the $\Delta_{k,n}^t$.

\begin{lemma}\label{lemma: del(b_n) = st(a_(n-1)) U  st(b_(n-1))}
    Let $\del\lno b_n \rno = \del_{\Delta_{k,n}^t}\lno b_n\rno$. Then,
    $$\del\lno b_n\rno = \st_{\del\lno b_n \rno}\lno a_{n-1}\rno \cup \st_{\del\lno b_n \rno}\lno b_{n-1}\rno.$$
\end{lemma}
\begin{proof}
Let $\sigma$ be a facet of $\operatorname{del}(b_n)$. 
If $a_{n-1} \in \sigma$, then $\sigma \in \operatorname{st}_{\operatorname{del}(b_n)}(a_{n-1})$, and the result follows. 
If $a_{n-1} \notin \sigma$, we claim that $b_{n-1} \in \sigma$; otherwise, both $a_{n-1}$ and $b_{n-1}$ would lie in the $k$-independent set contained in $\sigma^c$, contradicting the independence of that set. 
Hence $\sigma \in \operatorname{st}_{\operatorname{del}(b_n)}(b_{n-1})$, as required.
\end{proof}

Before presenting the main result, we state a preliminary result that will be useful in its proof. Recall that $\langle F\rangle$ denotes a simplex on the vertex set $F$. Also, for any graph $G$, we use $\Delta_0^t \left( G \right)$ to denote the simplicial complex whose only facet is the simplex with vertex set $V(G)$, i.e., $\Delta_0^t \left( G \right) = \langle V(G) \rangle$.

\begin{lemma}\label{claim: homotopy of total k-cut of 2×n gg}
Let $\Delta' = \del_{\Delta_{k,m+1}^t}\lno b_{m+1} \rno$. Then,
    \begin{enumerate}
        \item $\st_{\Delta'}\lno a_m \rno \cap \st_{\Delta'}\lno b_m \rno = \bigcup\limits_{i=1}^k K_i$, where $$K_1 = \left\langle \lcu a_m, a_{m+1}, b_m \rcu\right\rangle * \Delta_{k-1, m-1}^t,$$ and for $2\leq i \leq k$, $$K_i = \left\langle \lcu a_{m-i+1}, b_{m-i+1} \rcu\right\rangle * \Delta_{k-i, m-i}^t.$$ 

        \item $K_i \cap K_j \subseteq K_1$, for all $2 \leq i < j \leq k$. 
    \end{enumerate}  
\end{lemma}
    \begin{proof}
    
    \begin{enumerate}
        \item To prove this, we first show that 
$\st_{\Delta'}\lno a_m \rno \cap \st_{\Delta'}\lno b_m \rno \supseteq \bigcup\limits_{i=1}^k K_i.$
For this, let $\sigma \in \bigcup\limits_{i=1}^k K_i$. If $\sigma \in K_1$, the claim follows immediately. 
Assume now that $\sigma \in K_i$ for some $i > 1$. We proceed by constructing appropriate simplices depending on whether $i$ is even or odd:

\begin{enumerate}
    \item If $i$ is even, define $\tau_1$ and $\tau_2$ as follows:
    \begin{align*}
        \tau_1 &= \sigma \cup \lcu a_{m-i+2}, b_{m-i+3}, a_{m-i+4}, \dots, b_{m-1}, a_m \rcu,\\
        \tau_2 &= \sigma \cup \lcu b_{m-i+2}, a_{m-i+3}, b_{m-i+4}, \dots, a_{m-1}, b_m, a_{m+1} \rcu;
    \end{align*}

    \item If $i$ is odd, define $\tau_3$ and $\tau_4$ as follows:
    \begin{align*}
        \tau_3 &= \sigma \cup \lcu a_{m-i+2}, b_{m-i+3}, a_{m-i+4}, \dots, a_{m-1}, b_m, a_{m+1} \rcu,\\
        \tau_4 &= \sigma \cup \lcu b_{m-i+2}, a_{m-i+3}, b_{m-i+4}, \dots, b_{m-1}, a_m \rcu.
    \end{align*}
\end{enumerate}

It is easy to verify that $\tau_l \in \Delta'$ for all $1 \leq l \leq 4$. Consequently, both $\sigma \cup \lcu a_m \rcu \in \Delta'$ and $\sigma \cup \lcu b_m \rcu \in \Delta'$ hold, which implies that $\sigma \in \st_{\Delta'}\lno a_m \rno \cap \st_{\Delta'}\lno b_m \rno.$

To show that $\st_{\Delta'}\lno a_m \rno \cap \st_{\Delta'}\lno b_m \rno \subseteq \bigcup\limits_{i=1}^k K_i$,
let $\sigma \in \st_{\Delta'}\lno a_m \rno \cap \st_{\Delta'}\lno b_m \rno$. 
Without loss of generality (WLOG), assume that $\sigma$ is a facet of $ \st_{\Delta'}\lno a_m \rno \cap \st_{\Delta'}\lno b_m \rno$. 

If $a_{m+1} \in \sigma$, the fact that $\sigma \cup \{a_m\} \in \Delta'$ implies 
$\sigma \cup \{a_m,b_m\} \in \Delta'$, so $\sigma \in K_1$. 
Now consider the case $a_{m+1} \notin \sigma$. 

Let $C_i = \lcu a_{m-i+1}, b_{m-i+1} \rcu$ for $1 \le i \le k$, 
i.e., $C_i$ denotes the set of vertices in the $(i+1)$th column of $\mathcal{G}_{2\times(m+1)}$ from the right. 
If $\sigma \cap C_i = \emptyset$ for all $1 \le i \le k$, then $\sigma \in K_1$. 
If instead $C_i \subseteq \sigma$ for some smallest $1 \le i \le k$, then $\sigma \in K_i$. 

The remaining possibility is that $\sigma$ contains exactly one vertex from some $C_i$. 
We claim that this case cannot occur. 
Let $i$ be the smallest index such that $a_{m-i+1} \in \sigma$ and $b_{m-i+1} \notin \sigma$; 
equivalently, $\sigma \cap \lcu a_{m-j+1}, b_{m-j+1} \rcu = \emptyset$ for all $j < i$. 
We divide the argument into three cases, showing in each that 
$\sigma \cup \{b_{m-i+1}\} \in \st_{\Delta'}\lno a_m \rno \cap \st_{\Delta'}\lno b_m \rno$, 
contradicting the assumption that $\sigma$ is a facet of $\st_{\Delta'}\lno a_m \rno \cap \st_{\Delta'}\lno b_m \rno$.
    
    \begin{enumerate}
        \item \textbf{Case $i=1$:} 
        Since $a_m \in \sigma$ and $\sigma \in \st_{\Delta'}\lno b_m \rno$, it follows that 
        $\sigma \cup \{b_m\} \in \st_{\Delta'}\lno a_m \rno \cap \st_{\Delta'}\lno b_m \rno$.
    
        \item \textbf{Case $i$ even:} 
        Let $S_1$ and $S_2$ be two $k$-independent sets in $(\sigma \cup \{a_m\})^c$ and $(\sigma \cup \{b_m\})^c$, respectively. 
        In this and the next case, $A^c$ denotes the set $V(\mathcal{G}_{2\times (m+1)}) \setminus A$. 
        We may assume that $b_m \in S_1$ and $a_m \in S_2$; otherwise, 
        $\sigma \cup \{a_m,b_m\}$ would be a face of $\Delta'$, contradicting the assumption that 
        $\sigma$ is a facet of $\st_{\Delta'}\lno a_m \rno \cap \st_{\Delta'}\lno b_m \rno$. 

        Our goal is to construct $S_1'$ and $S_2'$ in $(\sigma \cup \{a_m\})^c$ and $(\sigma \cup \{b_m\})^c$, respectively, 
        such that $b_{m-i+1} \notin S_1', S_2'$. 
        This will imply that $\sigma \cup \{b_{m-i+1}\} \in \st_{\Delta'}\lno a_m \rno \cap \st_{\Delta'}\lno b_m \rno$.

        We may further assume that $b_{m-i+1}$ belongs to both $S_1$ and $S_2$, as the same sets otherwise suffice. 
        Let $\mathcal{C} = \bigsqcup\limits_{j=1}^{i-1} C_j$, where $C_j = \lcu a_{m-j+1}, b_{m-j+1} \rcu$. 
        Since $b_m, b_{m-i+1} \in S_1$, with $i > 2$, we have $|S_1 \cap \mathcal{C}| \le i-2$. 
        Moreover, because $\sigma$ is a facet in $\st_{\Delta'}\lno a_m \rno \cap \st_{\Delta'}\lno b_m \rno$, equality holds, i.e., $|S_1 \cap \mathcal{C}| = i-2$. 

        Define $S_1'$ and $S_2'$ as follows:
        \begin{align*}
            S_1' &= \lno S_1 \setminus (\{b_{m-i+1}\} \cup (S_1         \cap \mathcal{C})) \rno 
                \cup \{b_{m-i+2}, a_{m-i+3}, b_{m-i+4}, \dots, a_{m-1}, b_m\},\\[4pt]
            S_2' &= \lno S_1 \setminus (\{b_{m-i+1}, a_{m+1}\}          \cup (S_1 \cap \mathcal{C})) \rno 
                \cup \{a_{m-i+2}, b_{m-i+3}, a_{m-i+4}, \dots, b_{m-1}, a_m, b_{m+1}\}.
        \end{align*}
        Then $S_1'$ and $S_2'$ are $k$-independent sets in $(\sigma \cup \{a_m\})^c$ and $(\sigma \cup \{b_m\})^c$, respectively, 
        completing the argument for the case when $i$ is even.
        
        \item \textbf{Case $i>1$ odd:}
            The argument proceeds similarly to the even case, with the only difference being the construction of the sets ${S}_1^\prime$ and ${S}_2^\prime$. Here, we use the set $S_2$ to define ${S}_1^\prime$ and ${S}_2^\prime$. 
    
            Since $a_m, b_{m-i+1}\in S_2$, we have $|S_2 \cap \mathcal{C}|\leq i-2$. Because $\sigma$ is a facet of $\st_{\Delta'}\lno a_m \rno \cap \st_{\Delta'}\lno b_m \rno$, equality must hold, i.e., $|S_2 \cap \mathcal{C}|= i-2$.  Define,
            \begin{align*}
               {S}_1^\prime & = \lno S_2\setminus (\{b_{m-i+1},b_{m+1}\} \cup (S_2\cap \mathcal{C})) \rno  \cup \{a_{m-i+2},b_{m-i+3},a_{m-i+4},\dots,a_{m-1},b_m,a_{n+1}\},\\
                {S}_2^\prime & = \lno S_2\setminus (\{b_{m-i+1},b_{m+1}\} \cup (S_2\cap \mathcal{C})) \rno  \cup \{b_{m-i+2},a_{m-i+3},b_{m-i+4},\dots,b_{m-1},a_m\}.
            \end{align*}
            The rest of the argument follows as in the even case.
    \end{enumerate}
        Consequently, $\sigma \in \bigcup\limits_{i=0}^k K_i$, which completes the proof of \Cref{claim: homotopy of total k-cut of 2×n gg}(1).

    \item First observe that for $2 \le i < j \le k$ we have
\[
K_i \cap K_j 
  = \lsq \Delta_{k-i,\,m-i}^t \rsq 
    \cap 
    \lsq \langle \lno a_{m-j+1}, b_{m-j+1} \rno \rangle * \Delta_{k-j,\,m-j}^t \rsq.
\]
Indeed, $a_{m-i+1}$ and $b_{m-i+1}$ never appear in any face of $K_j$. 
Moreover, if $\sigma \in K_i \cap K_j$, then $\sigma$ contains no vertex to the right of $C_j$, and exactly one vertex of $C_j$ may lie in the $k$-independent set contained in $\sigma^c$. 
This determines the intersection to be the complex $\Delta_{k-j+1,\,m-j+1}^t$, i.e.,
\[
K_i \cap K_j = \Delta_{k-j+1,\,m-j+1}^t.
\]

Since $\Delta_{k-j+1,\,m-j+1}^t \subseteq K_1$, we obtain $K_i \cap K_j \subseteq K_1$ for all such $i$ and $j$.
\end{enumerate}

This completes the proof of \Cref{claim: homotopy of total k-cut of 2×n gg}.

\end{proof}

    We now prove the main result of this section.

\begin{theorem}\label{thm: homotopy type of total k-cut complex of 2n-gg}
    For $2 \le k \le n$, the total $k$-cut complex of $\twogg$, denoted $\Delta_{k,n}^t$, is homotopy equivalent to a wedge of $\binom{n-1}{k-1}$ spheres of dimension $\lno 2n-2k\rno$, i.e.,
    \[
        \Delta_{k,n}^t \simeq \bigvee_{\binom{n-1}{k-1}} \bbS^{\lno 2n-2k\rno}.
    \]
\end{theorem}

We prove the result by induction on $n$. 
For the inductive step, assume that the statement holds for all $n \le m$. 
To establish the case $n = m+1$, we analyze the link and deletion of the vertex $b_{m+1}$ in $\Delta_{k,m+1}^t$. 
The homotopy type of the link of $b_{m+1}$ is determined in \Cref{link of b_(m+1)}, and the homotopy type of the deletion of $b_{m+1}$ is computed in \Cref{deletion of b_n+1} using \Cref{lemma: bjorner's lemma 10.4(b)}, 
\Cref{lemma: del(b_n) = st(a_(n-1)) U  st(b_(n-1))}, 
and \Cref{claim: homotopy of total k-cut of 2×n gg}. 
Finally, we show that the link of $b_{m+1}$ is contractible inside the deletion of $b_{m+1}$, and \Cref{lemma: finding homotopy using link and deletion} then yields the desired conclusion.

\begin{proof}
    Although our argument applies for all $2 \le k \le n$, we note that the case $k=2$ also follows from \cite[Theorem~4.16]{bayer2024total}. 
    
    We proceed by induction on $n$. The statement is immediate when $n=2$. 
    Assume that the result holds for all $n \le m$, and consider the case $n=m+1$. 
    
    When $k=m+1$, the complex $\Delta_{k,m+1}^t$ consists of the following two facets, depending on whether $m$ is even or odd:
    \begin{enumerate}
        \item If $m$ is even,
        \[
        \Delta_{m+1,m+1}^t = \left\langle
            \lcu a_1, b_2, a_3, b_4,\dots, a_{m-1}, b_m, a_{m+1} \rcu,\,
            \lcu b_1, a_2, b_3, a_4,\dots, b_{m-1}, a_m, b_{m+1} \rcu
        \right\rangle;
        \]

        \item If $m$ is odd,
        \[
        \Delta_{m+1,m+1}^t = \left\langle
            \lcu a_1, b_2, a_3, b_4,\dots, b_{m-1}, a_m, b_{m+1} \rcu,\,
            \lcu b_1, a_2, b_3, a_4,\dots, a_{m-1}, b_m, a_{m+1} \rcu
        \right\rangle.
        \]
    \end{enumerate}
    
    In each case, $\Delta_{m+1,m+1}^t$ is homotopy equivalent to $\bbS^{0}$. Hence, the claim holds when $k=m+1$.

    Now let $2 \le k \le m$. 
    In this case, we determine the homotopy types of the link and deletion of the vertex $b_{m+1}$ in $\Delta_{k,m+1}^t$, and then apply \Cref{lemma: finding homotopy using link and deletion}.
    
    \medskip
    
    \noindent\textbf{Link of the vertex $b_{m+1}$ in $\Delta_{k,m+1}^t$:}
    By \Cref{coro:link of a vertex},
    \[
    \lk_{\Delta_{k,m+1}^t}\lno b_{m+1}\rno 
        = \totalkcut\lno \mathcal{G}_{2\times m}' \rno.
    \]
    
    The vertex $a_{m+1}$ is a leaf in $\mathcal{G}_{2\times m}'$ 
    (see \Cref{fig: graph of G'(2×n)}), and 
    \Cref{thm: suspension of previous graph} gives
    \[
    \totalkcut\lno \mathcal{G}_{2\times m}' \rno 
        \simeq 
        \Sigma\!\left( \totalkcut\lno \mathcal{G}_{2\times m} \rno \right).
    \]
    
    The induction hypothesis then yields
    \begin{equation}\label{link of b_n+1}
        \lk_{\Delta_{k,m+1}^t}\lno b_{m+1}\rno 
        \simeq 
        \bigvee_{\binom{m-1}{k-1}} \bbS^{\lno 2m-2k+1\rno}.
\end{equation}

    \noindent\textbf{Deletion of the vertex $b_{m+1}$ in $\Delta_{k,m+1}^t$:} 
    Recall that $\Delta' = \del_{\Delta_{k,m+1}^t}\lno b_{m+1} \rno$. 
    By \Cref{lemma: del(b_n) = st(a_(n-1)) U  st(b_(n-1))},
    \[
    \Delta' = \st_{\Delta'}\lno a_m \rno \cup \st_{\Delta'}\lno b_m \rno.
    \]
    
    Since both $\st_{\Delta'}\lno a_m \rno$ and $\st_{\Delta'}\lno b_m \rno$ are contractible,
    \Cref{lemma: bjorner's lemma 10.4(b)} gives
    \[
    \Delta' \simeq \Sigma\!\left( \st_{\Delta'}\lno a_m \rno \cap 
                                    \st_{\Delta'}\lno b_m \rno \right).
    \]
    
    It remains to determine the homotopy type of 
    $\st_{\Delta'}\lno a_m \rno \cap \st_{\Delta'}\lno b_m \rno$. 
    Each $K_i$ defined in 
    \Cref{claim: homotopy of total k-cut of 2×n gg} is contractible for 
    $1 \le i \le k$. 
    Using \Cref{claim: homotopy of total k-cut of 2×n gg} together with 
    \Cref{lemma: bjorner's lemma 10.4(b)}, we obtain
    \[
    \st_{\Delta'}\lno a_m \rno \cap \st_{\Delta'}\lno b_m \rno 
        \simeq 
        \bigvee_{2 \le i \le k} 
            \Sigma\!\left( K_1 \cap K_i \right).
    \]

    To determine the homotopy type of 
    $\st_{\Delta'}\lno a_m \rno \cap \st_{\Delta'}\lno b_m \rno$, it suffices to compute the homotopy types of $K_1 \cap K_i$ for 
    $2 \le i \le k$. 
    Using arguments analogous to those used for $K_i \cap K_j$ in 
    \cref{claim: homotopy of total k-cut of 2×n gg}, we obtain
    \[
    K_1 \cap K_i 
        = \Delta_{k-i+1}^t\!\left( \mathcal{G}_{2\times (m-i+1)} \right).
    \]
    
    By the induction hypothesis,
    \[
    K_1 \cap K_i 
        = \bigvee_{\binom{m-i}{\,k-i\,}} 
            \bbS^{\lno 2m - 2k \rno}.
    \]
    
    Consequently,
    \[
    \st_{\Delta'}\lno a_m \rno \cap \st_{\Delta'}\lno b_m \rno
        \simeq 
        \bigvee_{N} \bbS^{\lno 2m - 2k + 1 \rno},
    \]
    where $N = \binom{m-2}{k-2}+\binom{m-3}{k-3}+\dots+\binom{m-k}{k-k}$.
    From \Cref{prop: binomial result}, $N = \binom{m-1}{k-2}$. 
    Therefore,
    \begin{equation}\label{deletion of b_n+1}
        \Delta' 
        = \del_{\Delta_{k,m+1}^t}\lno b_{m+1} \rno 
        \simeq 
        \bigvee_{\binom{m-1}{k-2}} 
            \bbS^{\lno 2m - 2k + 2 \rno}.
    \end{equation}

From \Cref{link of b_n+1}, the link of $b_{m+1}$ in $\Delta_{k,m+1}^t$ is homotopy equivalent to a wedge of spheres of dimension $2m - 2k + 1$, and \Cref{deletion of b_n+1} shows that the deletion of $b_{m+1}$ in $\Delta_{k,m+1}^t$ is homotopy equivalent to a wedge of spheres of dimension $2m - 2k + 2$. 
Since the former is contractible inside the latter, 
\Cref{lemma: finding homotopy using link and deletion} yields
\[
\Delta_{k,m+1}^t 
    \simeq 
    \del_{\Delta_{k,m+1}^t}\lno b_{m+1} \rno 
    \,\vee\, 
    \Sigma\!\left( \lk_{\Delta_{k,m+1}^t}\lno b_{m+1}\rno \right).
\]

\noindent Consequently,
\begin{align*}
    \Delta_{k,m+1}^t 
        &\simeq 
            \left( \bigvee_{\binom{m-1}{k-2}} 
                    \bbS^{\lno 2m - 2k + 2 \rno} \right)
            \,\vee\,
            \Sigma\!\left( 
                \bigvee_{\binom{m-1}{k-1}} 
                \bbS^{\lno 2m - 2k + 1 \rno}
            \right)\\[4pt]
        &\simeq 
            \bigvee_{\binom{m-1}{k-2} + \binom{m-1}{k-1}} 
            \bbS^{\lno 2m - 2k + 2 \rno}\\[4pt]
        &\simeq 
            \bigvee_{\binom{m}{k-1}} 
            \bbS^{\lno 2m - 2k + 2 \rno},
\end{align*}
using the identity 
$\binom{m}{k-1} = \binom{m-1}{k-2} + \binom{m-1}{k-1}$. 

This completes the proof of 
\Cref{thm: homotopy type of total k-cut complex of 2n-gg}.
\end{proof}

\section{Total 3-cut complex of $\left( 3 \times n \right) $ grid graph.}\label{section:total cut 3*k}

In this section, we determine the homotopy type of the total $3$-cut complex of $\threegg$, i.e., $\totalthreecut \lno \threegg \rno$ for $n \ge 2$. 
For convenience, we denote $\totalthreecut\lno \threegg \rno$ by $\Delta_n^t$ throughout this section.
The vertex and edge sets of $\threegg$ (see \Cref{fig: graph of G(3×n)}) are given as follows:

\begin{align*}
    V\lno \threegg\rno = &\lcu a_i, b_i, c_i\ \middle|\ 1\leq i \leq n\rcu;\\
    E\lno \threegg\rno = &\lcu \lno a_i, a_{i+1}\rno, \lno b_i, b_{i+1}\rno, \lno c_i, c_{i+1}\rno\in V\lno \threegg\rno\times V\lno \threegg\rno\ \middle|\ 1\leq i \leq n-1\rcu\\& \sqcup \lcu \lno a_j, b_{j}\rno, \lno b_j, c_j\rno \in V\lno \twogg\rno\times V\lno \twogg\rno\ \middle|\ 1\leq j \leq n\rcu.
\end{align*}

\begin{figure}[h!]
    \centering
    \begin{tikzpicture}[line cap=round,line join=round,>=triangle 45,x=0.4cm,y=0.4cm]
        \clip(-1,-2) rectangle (14.5,7);
        \draw [line width=1pt] (0,0)-- (0,3);
        \draw [line width=1pt] (0,3)-- (0,6);
        \draw [line width=1pt] (3,0)-- (3,3);
        \draw [line width=1pt] (3,3)-- (3,6);
        \draw [line width=1pt] (6,0)-- (6,3);
        \draw [line width=1pt] (6,3)-- (6,6);
        \draw [line width=1pt] (6,3)-- (7,3);
        \draw [line width=1pt] (6,6)-- (7,6);
        \draw [line width=1pt] (6,0)-- (7,0);
        \draw [line width=1pt] (0,3)-- (6,3);
        \draw [line width=1pt] (0,0)-- (6,0);
        \draw [line width=1pt] (0,6)-- (6,6);
        \draw [line width=1pt,dash pattern=on 4pt off 4pt] (7,6)-- (9,6);
        \draw [line width=1pt,dash pattern=on 4pt off 4pt] (7,3)-- (9,3);
        \draw [line width=1pt,dash pattern=on 4pt off 4pt] (7,0)-- (9,0);
        \draw [line width=1pt] (9,3)-- (10,3);
        \draw [line width=1pt] (9,6)-- (10,6);
        \draw [line width=1pt] (9,0)-- (10,0);
        \draw [line width=1pt] (10,0)-- (10,3);
        \draw [line width=1pt] (10,3)-- (10,6);
        \draw [line width=1pt] (13,0)-- (13,3);
        \draw [line width=1pt] (13,3)-- (13,6);
        \draw [line width=1pt] (10,6)-- (13,6);
        \draw (-0.2,7.2) node[anchor=north west] {\large$a_{_1}$};
        \draw (12.8,7.1) node[anchor=north west] {\large$a_{_{n}}$};
        \draw (9.8,7.2) node[anchor=north west] {\large$a_{_{n-1}}$};
        \draw (5.8,7.2) node[anchor=north west] {\large$a_{_3}$};
        \draw (2.8,7.2) node[anchor=north west] {\large$a_{_2}$};
        \draw (-0.2,4.5) node[anchor=north west] {\large$b_{_1}$};
        \draw (12.8,4.4) node[anchor=north west] {\large$b_{_{n}}$};
        \draw (9.8,4.5) node[anchor=north west] {\large$b_{_{n-1}}$};
        \draw (5.8,4.5) node[anchor=north west] {\large$b_{_3}$};
        \draw (2.8,4.5) node[anchor=north west] {\large$b_{_2}$};
        \draw (-0.2,1.2) node[anchor=north west] {\large$c_{_1}$};
        \draw (12.8,1.1) node[anchor=north west] {\large$c_{_{n}}$};
        \draw (9.8,1.2) node[anchor=north west] {\large$c_{_{n-1}}$};
        \draw (5.8,1.2) node[anchor=north west] {\large$c_{_3}$};
        \draw (2.8,1.2) node[anchor=north west] {\large$c_{_2}$};
        \draw [line width=1pt] (10,3)-- (13,3);
        \draw [line width=1pt] (10,0)-- (13,0);
        \begin{scriptsize}
            \draw [fill=wwwwww] (0,0) circle (2pt);
            \draw [fill=wwwwww] (0,3) circle (2pt);
            \draw [fill=wwwwww] (0,6) circle (2pt);
            \draw [fill=wwwwww] (3,0) circle (2pt);
            \draw [fill=wwwwww] (3,3) circle (2pt);
            \draw [fill=wwwwww] (3,6) circle (2pt);
            \draw [fill=wwwwww] (6,0) circle (2pt);
            \draw [fill=wwwwww] (6,3) circle (2pt);
            \draw [fill=wwwwww] (6,6) circle (2pt);
            \draw [fill=wwwwww] (10,3) circle (2pt);
            \draw [fill=wwwwww] (10,6) circle (2pt);
            \draw [fill=wwwwww] (10,0) circle (2pt);
            \draw [fill=wwwwww] (13,0) circle (2pt);
            \draw [fill=wwwwww] (13,3) circle (2pt);
            \draw [fill=wwwwww] (13,6) circle (2pt);
        \end{scriptsize}
    \end{tikzpicture}
    \vspace{-2em}
    \caption{The graph $\threegg$.}
    \label{fig: graph of G(3×n)}
\end{figure}

We also define the following graphs, which will play a role in the arguments later in this section.
\begin{enumerate}[label = \arabic*.]
\itemsep0.5em
    \item $H_1 = \mthreeggn \lsq V \lno \mthreeggn \rno \setminus \lcu a_{m+1}\rcu \rsq$ (see \Cref{fig: graph of G'(3×n)});
    \item $H_2 = \mthreeggn \lsq V \lno \mthreeggn\rno \setminus \lcu  a_m,b_m,a_{m+1},b_{m+1} \rcu \rsq$ (see \Cref{fig: graph of G(1)(3×n)});
    \item $H_3 = \mthreegg\lsq V\lno \mthreegg \rno \setminus \lcu b_m, c_m \rcu \rsq$ (see \Cref{fig: graph of H1});
    \item $H_4 = \mathcal{G}_{3\times \lno m-1 \rno}\lsq V \lno \mathcal{G}_{3\times \lno m-1 \rno} \rno \setminus \lcu b_{m-1}, c_{m-1} \rcu \rsq$ (see \Cref{fig: graph of H2}).
\end{enumerate}

\begin{figure}[H]
    \begin{minipage}{0.5\textwidth}
        \centering
        \begin{tikzpicture}[line cap=round,line join=round,>=triangle 45,x=0.4cm,y=0.4cm]
            \clip(-1,-2) rectangle (18,7);
            \draw [line width=1pt] (0,0)-- (0,3);
            \draw [line width=1pt] (0,3)-- (0,6);
            \draw [line width=1pt] (3,0)-- (3,3);
            \draw [line width=1pt] (3,3)-- (3,6);
            \draw [line width=1pt] (6,0)-- (6,3);
            \draw [line width=1pt] (6,3)-- (6,6);
            \draw [line width=1pt] (6,3)-- (7,3);
            \draw [line width=1pt] (6,6)-- (7,6);
            \draw [line width=1pt] (6,0)-- (7,0);
            \draw [line width=1pt] (0,3)-- (6,3);
            \draw [line width=1pt] (0,0)-- (6,0);
            \draw [line width=1pt] (0,6)-- (6,6);
            \draw [line width=1pt,dash pattern=on 4pt off 4pt] (7,6)-- (9,6);
            \draw [line width=1pt,dash pattern=on 4pt off 4pt] (7,3)-- (9,3);
            \draw [line width=1pt,dash pattern=on 4pt off 4pt] (7,0)-- (9,0);
            \draw [line width=1pt] (9,3)-- (10,3);
            \draw [line width=1pt] (9,6)-- (10,6);
            \draw [line width=1pt] (9,0)-- (10,0);
            \draw [line width=1pt] (10,0)-- (10,3);
            \draw [line width=1pt] (10,3)-- (10,6);
            \draw [line width=1pt] (13,0)-- (13,3);
            \draw [line width=1pt] (16,0)-- (16,3);
            \draw [line width=1pt] (13,3)-- (16,3);
            \draw [line width=1pt] (13,3)-- (13,6);
            \draw [line width=1pt] (10,6)-- (13,6);
            \draw (-0.2,7.2) node[anchor=north west] {\large$a_{_1}$};
            \draw (12.8,7.1) node[anchor=north west] {\large$a_{_{m}}$};
            \draw (9.8,7.2) node[anchor=north west] {\large$a_{_{m-1}}$};
            \draw (5.8,7.2) node[anchor=north west] {\large$a_{_3}$};
            \draw (2.8,7.2) node[anchor=north west] {\large$a_{_2}$};
            \draw (-0.2,4.5) node[anchor=north west] {\large$b_{_1}$};
            \draw (12.8,4.5) node[anchor=north west] {\large$b_{_{m}}$};
            \draw (9.8,4.5) node[anchor=north west] {\large$b_{_{m-1}}$};      
            \draw (15.8,4.4) node[anchor=north west] {\large$b_{_{m+1}}$};
            \draw (5.8,4.5) node[anchor=north west] {\large$b_{_3}$};
            \draw (2.8,4.5) node[anchor=north west] {\large$b_{_2}$};
            \draw (-0.2,1.2) node[anchor=north west] {\large$c_{_1}$};
            \draw (12.8,1.2) node[anchor=north west] {\large$c_{_{m}}$};
            \draw (15.8,1.2) node[anchor=north west] {\large$c_{_{m+1}}$};
            \draw (9.8,1.2) node[anchor=north west] {\large$c_{_{m-1}}$};
            \draw (5.8,1.2) node[anchor=north west] {\large$c_{_3}$};
            \draw (2.8,1.2) node[anchor=north west] {\large$c_{_2}$};
            \draw [line width=1pt] (10,3)-- (13,3);
            \draw [line width=1pt] (10,0)-- (13,0);
            \draw [line width=1pt] (13,0)-- (16,0);
            \begin{scriptsize}
                \draw [fill=wwwwww] (0,0) circle (2pt);
                \draw [fill=wwwwww] (0,3) circle (2pt);
                \draw [fill=wwwwww] (0,6) circle (2pt);
                \draw [fill=wwwwww] (3,0) circle (2pt);
                \draw [fill=wwwwww] (3,3) circle (2pt);
                \draw [fill=wwwwww] (3,6) circle (2pt);
                \draw [fill=wwwwww] (6,0) circle (2pt);
                \draw [fill=wwwwww] (6,3) circle (2pt);
                \draw [fill=wwwwww] (6,6) circle (2pt);
                \draw [fill=wwwwww] (10,3) circle (2pt);
                \draw [fill=wwwwww] (10,6) circle (2pt);
                \draw [fill=wwwwww] (10,0) circle (2pt);
                \draw [fill=wwwwww] (13,0) circle (2pt);
                \draw [fill=wwwwww] (16,0) circle (2pt);
                \draw [fill=wwwwww] (16,3) circle (2pt);
                \draw [fill=wwwwww] (13,3) circle (2pt);
                \draw [fill=wwwwww] (13,6) circle (2pt);
            \end{scriptsize}
            \end{tikzpicture}
        \vspace{-2em}
        \caption{The graph $H_1$.}
        \label{fig: graph of G'(3×n)}
    \end{minipage}%
    ~
    \begin{minipage}{0.5\textwidth}
    \centering
        \begin{tikzpicture}[scale=0.9,line cap=round,line join=round,>=triangle 45,x=0.4cm,y=0.4cm]
        \clip(-1,-2) rectangle (21,7);
        \draw [line width=1pt] (0,0)-- (0,3);
        \draw [line width=1pt] (0,3)-- (0,6);
        \draw [line width=1pt] (3,0)-- (3,3);
        \draw [line width=1pt] (3,3)-- (3,6);
        \draw [line width=1pt] (6,0)-- (6,3);
        \draw [line width=1pt] (6,3)-- (6,6);
        \draw [line width=1pt] (6,3)-- (7,3);
        \draw [line width=1pt] (6,6)-- (7,6);
        \draw [line width=1pt] (6,0)-- (7,0);
        \draw [line width=1pt] (0,3)-- (6,3);
        \draw [line width=1pt] (0,0)-- (6,0);
        \draw [line width=1pt] (0,6)-- (6,6);
        \draw [line width=1pt,dash pattern=on 4pt off 4pt] (7,6)-- (9,6);
        \draw [line width=1pt,dash pattern=on 4pt off 4pt] (7,3)-- (9,3);
        \draw [line width=1pt,dash pattern=on 4pt off 4pt] (7,0)-- (9,0);
        \draw [line width=1pt] (9,3)-- (10,3);
        \draw [line width=1pt] (9,6)-- (10,6);
        \draw [line width=1pt] (9,0)-- (10,0);
        \draw [line width=1pt] (10,0)-- (10,3);
        \draw [line width=1pt] (10,3)-- (10,6);
        \draw [line width=1pt] (13,0)-- (13,3);
        \draw [line width=1pt] (13,3)-- (13,6);
        \draw [line width=1pt] (10,6)-- (13,6);
        \draw (-0.2,7.2) node[anchor=north west] {\large$a_{_1}$};
        \draw (12.8,7.2) node[anchor=north west] {\large$a_{_{m-1}}$};
        \draw (9.8,7.35) node[anchor=north west] {\large$a_{_{m-2}}$};
        \draw (5.8,7.35) node[anchor=north west] {\large$a_{_3}$};
        \draw (2.8,7.35) node[anchor=north west] {\large$a_{_2}$};
        \draw (-0.2,4.6) node[anchor=north west] {\large$b_{_1}$};
        \draw (12.8,4.5) node[anchor=north west] {\large$b_{_{m-1}}$};
        \draw (9.8,4.6) node[anchor=north west] {\large$b_{_{m-2}}$};
        \draw (5.8,4.6) node[anchor=north west] {\large$b_{_3}$};
        \draw (2.8,4.6) node[anchor=north west] {\large$b_{_2}$};
        \draw (-0.2,1.35) node[anchor=north west] {\large$c_{_1}$};
        \draw (12.8,1.35) node[anchor=north west] {\large$c_{_{m-1}}$};
        \draw (15.7,1.35) node[anchor=north west] {\large$c_{_{m}}$};
        \draw (18.6,1.25) node[anchor=north west] {\large$c_{_{m+1}}$};
        \draw (9.8,1.35) node[anchor=north west] {\large$c_{_{m-2}}$};
        \draw (5.8,1.35) node[anchor=north west] {\large$c_{_3}$};
        \draw (2.8,1.35) node[anchor=north west] {\large$c_{_2}$};
        \draw [line width=1pt] (10,3)-- (13,3);
        \draw [line width=1pt] (10,0)-- (13,0);
        
        \draw [line width=1pt] (13,0)-- (19,0);
        \begin{scriptsize}
        \draw [fill=wwwwww] (0,0) circle (2pt);
        \draw [fill=wwwwww] (0,3) circle (2pt);
        \draw [fill=wwwwww] (0,6) circle (2pt);
        \draw [fill=wwwwww] (3,0) circle (2pt);
        \draw [fill=wwwwww] (3,3) circle (2pt);
        \draw [fill=wwwwww] (3,6) circle (2pt);
        \draw [fill=wwwwww] (6,0) circle (2pt);
        \draw [fill=wwwwww] (6,3) circle (2pt);
        \draw [fill=wwwwww] (6,6) circle (2pt);
        \draw [fill=wwwwww] (10,3) circle (2pt);
        \draw [fill=wwwwww] (10,6) circle (2pt);
        \draw [fill=wwwwww] (10,0) circle (2pt);
        \draw [fill=wwwwww] (13,0) circle (2pt);
        \draw [fill=wwwwww] (16,0) circle (2pt);
        \draw [fill=wwwwww] (19,0) circle (2pt);
        \draw [fill=wwwwww] (13,3) circle (2pt);
        \draw [fill=wwwwww] (13,6) circle (2pt);
        \end{scriptsize}
        \end{tikzpicture}
        \vspace{-3em}
        \caption{The graph $H_2$.}
        \label{fig: graph of G(1)(3×n)}
    \end{minipage}   
\end{figure}

\begin{figure}[H]
    \begin{minipage}{0.5\textwidth}
    \centering
        \begin{tikzpicture}[scale=0.9, line cap=round,line join=round,>=triangle 45,x=0.4cm,y=0.4cm]
            \clip(-1,-2) rectangle (17.5,7);
            \draw [line width=1pt] (0,0)-- (0,3);
            \draw [line width=1pt] (0,3)-- (0,6);
            \draw [line width=1pt] (3,0)-- (3,3);
            \draw [line width=1pt] (3,3)-- (3,6);
            \draw [line width=1pt] (6,0)-- (6,3);
            \draw [line width=1pt] (6,3)-- (6,6);
            \draw [line width=1pt] (6,3)-- (7,3);
            \draw [line width=1pt] (6,6)-- (7,6);
            \draw [line width=1pt] (6,0)-- (7,0);
            \draw [line width=1pt] (0,3)-- (6,3);
            \draw [line width=1pt] (0,0)-- (6,0);
            \draw [line width=1pt] (0,6)-- (6,6);
            \draw [line width=1pt,dash pattern=on 4pt off 4pt] (7,6)-- (9,6);
            \draw [line width=1pt,dash pattern=on 4pt off 4pt] (7,3)-- (9,3);
            \draw [line width=1pt,dash pattern=on 4pt off 4pt] (7,0)-- (9,0);
            \draw [line width=1pt] (9,3)-- (10,3);
            \draw [line width=1pt] (9,6)-- (10,6);
            \draw [line width=1pt] (9,0)-- (10,0);
            \draw [line width=1pt] (10,0)-- (10,3);
            \draw [line width=1pt] (10,3)-- (10,6);
            \draw [line width=1pt] (13,0)-- (13,3);
            \draw [line width=1pt] (13,3)-- (13,6);
            \draw [line width=1pt] (10,6)-- (13,6);
            \draw (-0.3,7.3) node[anchor=north west] {\large$a_{_1}$};
            \draw (12.7,7.3) node[anchor=north west] {\large$a_{_{m-1}}$};
            \draw (9.7,7.3) node[anchor=north west] {\large$a_{_{m-2}}$};
            \draw (5.7,7.3) node[anchor=north west] {\large$a_{_3}$};
            \draw (2.7,7.3) node[anchor=north west] {\large$a_{_2}$};
            
            \draw (-0.2,4.6) node[anchor=north west] {\large$b_{_1}$};
            \draw (12.8,4.5) node[anchor=north west] {\large$b_{_{m-1}}$};
            \draw (9.8,4.6) node[anchor=north west] {\large$b_{_{m-2}}$};
            \draw (5.8,4.6) node[anchor=north west] {\large$b_{_3}$};
            \draw (2.8,4.6) node[anchor=north west] {\large$b_{_2}$};
            
            \draw (-0.2,1.3) node[anchor=north west] {\large$c_{_1}$};
            \draw (12.8,1.2) node[anchor=north west] {\large$c_{_{m-1}}$};
            \draw (15.8,7.3) node[anchor=north west] {\large$a_{_{m}}$};
            \draw (9.8,1.3) node[anchor=north west] {\large$c_{_{m-2}}$};
            \draw (5.8,1.3) node[anchor=north west] {\large$c_{_3}$};
            \draw (2.8,1.3) node[anchor=north west] {\large$c_{_2}$};
            \draw [line width=1pt] (10,3)-- (13,3);
            \draw [line width=1pt] (10,0)-- (13,0);
            \draw [line width=1pt] (13,6)-- (16,6);
            \begin{scriptsize}
                \draw [fill=wwwwww] (0,0) circle (2pt);
                \draw [fill=wwwwww] (0,3) circle (2pt);
                \draw [fill=wwwwww] (0,6) circle (2pt);
                \draw [fill=wwwwww] (3,0) circle (2pt);
                \draw [fill=wwwwww] (3,3) circle (2pt);
                \draw [fill=wwwwww] (3,6) circle (2pt);
                \draw [fill=wwwwww] (6,0) circle (2pt);
                \draw [fill=wwwwww] (6,3) circle (2pt);
                \draw [fill=wwwwww] (6,6) circle (2pt);
                \draw [fill=wwwwww] (10,3) circle (2pt);
                \draw [fill=wwwwww] (10,6) circle (2pt);
                \draw [fill=wwwwww] (10,0) circle (2pt);
                \draw [fill=wwwwww] (13,0) circle (2pt);
                \draw [fill=wwwwww] (16,6) circle (2pt);
                \draw [fill=wwwwww] (13,3) circle (2pt);
                \draw [fill=wwwwww] (13,6) circle (2pt);
            \end{scriptsize}
        \end{tikzpicture}
        \vspace{-2em}
        \caption{The graph $H_3$.}
        \label{fig: graph of H1}
    \end{minipage}%
    ~%\hfill
    \begin{minipage}{0.45\textwidth}
    \centering
        \begin{tikzpicture}[scale=0.9, line cap=round,line join=round,>=triangle 45,x=0.4cm,y=0.4cm]
            \clip(-1,-2) rectangle (18.5,7);
            \draw [line width=1pt] (0,0)-- (0,3);
            \draw [line width=1pt] (0,3)-- (0,6);
            \draw [line width=1pt] (3,0)-- (3,3);
            \draw [line width=1pt] (3,3)-- (3,6);
            \draw [line width=1pt] (6,0)-- (6,3);
            \draw [line width=1pt] (6,3)-- (6,6);
            \draw [line width=1pt] (6,3)-- (7,3);
            \draw [line width=1pt] (6,6)-- (7,6);
            \draw [line width=1pt] (6,0)-- (7,0);
            \draw [line width=1pt] (0,3)-- (6,3);
            \draw [line width=1pt] (0,0)-- (6,0);
            \draw [line width=1pt] (0,6)-- (6,6);
            \draw [line width=1pt,dash pattern=on 4pt off 4pt] (7,6)-- (9,6);
            \draw [line width=1pt,dash pattern=on 4pt off 4pt] (7,3)-- (9,3);
            \draw [line width=1pt,dash pattern=on 4pt off 4pt] (7,0)-- (9,0);
            \draw [line width=1pt] (9,3)-- (10,3);
            \draw [line width=1pt] (9,6)-- (10,6);
            \draw [line width=1pt] (9,0)-- (10,0);
            \draw [line width=1pt] (10,0)-- (10,3);
            \draw [line width=1pt] (10,3)-- (10,6);
            \draw [line width=1pt] (13,0)-- (13,3);
            \draw [line width=1pt] (13,3)-- (13,6);
            \draw [line width=1pt] (10,6)-- (13,6);
            \draw (-0.3,7.3) node[anchor=north west] {\large$a_{_1}$};
            \draw (12.7,7.3) node[anchor=north west] {\large$a_{_{m-2}}$};
            \draw (9.7,7.3) node[anchor=north west] {\large$a_{_{m-3}}$};
            \draw (5.7,7.3) node[anchor=north west] {\large$a_{_3}$};
            \draw (2.7,7.3) node[anchor=north west] {\large$a_{_2}$};
            
            \draw (-0.2,4.6) node[anchor=north west] {\large$b_{_1}$};
            \draw (12.8,4.5) node[anchor=north west] {\large$b_{_{m-2}}$};
            \draw (9.8,4.6) node[anchor=north west] {\large$b_{_{m-3}}$};
            \draw (5.8,4.6) node[anchor=north west] {\large$b_{_3}$};
            \draw (2.8,4.6) node[anchor=north west] {\large$b_{_2}$};
            
            \draw (-0.2,1.3) node[anchor=north west] {\large$c_{_1}$};
            \draw (12.8,1.2) node[anchor=north west] {\large$c_{_{m-2}}$};
            \draw (15.8,7.3) node[anchor=north west] {\large$a_{_{m-1}}$};
            \draw (9.8,1.3) node[anchor=north west] {\large$c_{_{m-3}}$};
            \draw (5.8,1.3) node[anchor=north west] {\large$c_{_3}$};
            \draw (2.8,1.3) node[anchor=north west] {\large$c_{_2}$};
            \draw [line width=1pt] (10,3)-- (13,3);
            \draw [line width=1pt] (10,0)-- (13,0);
            \draw [line width=1pt] (13,6)-- (16,6);
            \begin{scriptsize}
                \draw [fill=wwwwww] (0,0) circle (2pt);
                \draw [fill=wwwwww] (0,3) circle (2pt);
                \draw [fill=wwwwww] (0,6) circle (2pt);
                \draw [fill=wwwwww] (3,0) circle (2pt);
                \draw [fill=wwwwww] (3,3) circle (2pt);
                \draw [fill=wwwwww] (3,6) circle (2pt);
                \draw [fill=wwwwww] (6,0) circle (2pt);
                \draw [fill=wwwwww] (6,3) circle (2pt);
                \draw [fill=wwwwww] (6,6) circle (2pt);
                \draw [fill=wwwwww] (10,3) circle (2pt);
                \draw [fill=wwwwww] (10,6) circle (2pt);
                \draw [fill=wwwwww] (10,0) circle (2pt);
                \draw [fill=wwwwww] (13,0) circle (2pt);
                \draw [fill=wwwwww] (16,6) circle (2pt);
                \draw [fill=wwwwww] (13,3) circle (2pt);
                \draw [fill=wwwwww] (13,6) circle (2pt);
            \end{scriptsize}
        \end{tikzpicture}
        \vspace{-2em}
        \caption{The graph $H_4$.}
        \label{fig: graph of H2}
    \end{minipage}
\end{figure}

The following result will be useful in determining the homotopy type of $\Delta_n^t$.

\begin{lemma}\label{lemma: del(a_n) = st(a_(n-1)) U  st(b_(n-1))}
    Let $\del\lno a_{n} \rno = \del_{\Delta_n^t}\lno a_{n} \rno$. Then
    \[
        \del\lno a_{n} \rno 
        = \st_{\Delta_n^t}\lno a_{n-1} \rno 
          \cup 
          \st_{\Delta_n^t}\lno b_{n-1} \rno.
    \]
\end{lemma}

We omit the proof, as it follows the same argument as in 
\Cref{lemma: del(b_n) = st(a_(n-1)) U  st(b_(n-1))}.

We now determine the homotopy type of two auxiliary complexes ($\Delta_2^t \lno H_2 \rno$ and $\Delta_2^t \lno H_3 \rno$), arising from the grid graph, using a Morse matching. Our approach follows the argument in \cite[Theorem~4.16]{bayer2024total}. For this, we impose the following ordering on the vertices of $H_2$ and $H_3$:

\begin{itemize}
    \item If $i < j$, then $a_i < a_j$, $a_i < b_j$, $a_i < c_j$, $b_i < b_j$, $b_i < c_j$, $c_i < c_j$;
    \item If $i = j$, then $a_i < b_j$, $a_i < c_j$, $b_i < b_j$.
\end{itemize}
 
\begin{lemma}\label{thm: homotopy type of total 2-cut complex of (3×m)(1)}
    For $m \geq 2$, the total 2-cut complex of $H_2$ is homotopy equivalent to wedge of $\lno 2m-4 \rno-$ number of $\lno 3m-5 \rno-$dimensional spheres, i.e., $$\Delta_2^t \lno H_2 \rno \simeq \bigvee\limits_{2m-4}\bbS^{3m-5}.$$
\end{lemma}

\begin{proof}
    In $H_2$, the only neighbors of $a_{1}$ are $b_{1}$ and $a_{2}$, and 
    $\lvert V \lno H_2 \rno \rvert = 3m-1$ 
    (see \Cref{fig: graph of G(1)(3×n)}). 
    We have the following observations:

    \begin{enumerate}[label=\arabic*.]
        \item The number of edges in $H_2$ is 
        \(
            3(m-2) + 2(m-1) + 2 = 5m - 6,
        \)
        obtained by counting the horizontal and vertical edges.

        \item For every facet $\sigma \in \Delta_2^t \lno H_2 \rno$,  
        the set $\sigma$ is exactly 
        \( V\lno H_2 \rno \setminus \{x, y\} \), 
        where $\{x, y\}$ is a non-edge (a $2$-element independent set). 
        Hence, the number of non-facets of 
        $\Delta_2^t \lno H_2 \rno$ equals the number of edges of $H_2$, 
        namely $5m - 6$. 
        These non-facets are precisely the sets 
        \( V\lno H_2 \rno \setminus \{e_1, e_2\} \), 
        where $e_1$ and $e_2$ are the endpoints of some edge $e$.
    \end{enumerate}
    
    We first define an element matching on the face poset of $\Delta_2^t \lno H_2 \rno$ using the vertex $c_{m+1}$, denoted $M_{c_{m+1}}$. 
    It is immediate that $M_{c_{m+1}}$ matches every $\sigma \in \Delta_2^t \lno H_2 \rno$ with $c_{m+1} \notin \sigma$ to the face 
    $\sigma \cup \{c_{m+1}\}$, provided the latter belongs to $\Delta_2^t \lno H_2 \rno$. 
    Since $H_2$ is triangle-free, every subset of $V\lno H_2 \rno$ of size $3m-4$ lies in $\Delta_2^t \lno H_2 \rno$. 
    Thus, the faces that remain unmatched under $M_{c_{m+1}}$ must have sizes $3m-3$ or $3m-4$.
    
    A face $\sigma$ is unmatched exactly when  
    $c_{m+1} \notin \sigma$ and 
    $\sigma \cup \{c_{m+1}\} \notin \Delta_2^t \lno H_2 \rno$. 
    The latter condition holds precisely when  
    \(
    \sigma \cup \{c_{m+1}\} 
        = V\lno H_2 \rno \setminus \{e_1, e_2\},
    \)
    where $e_1$ and $e_2$ are the endpoints of some edge $e$ in $H_2$. 
    Hence, the unmatched faces are of the following two types:

    \begin{enumerate}[label=\textbf{Type }\Roman*:, leftmargin=4em]
    \item Subsets of size $3m-4$, namely
    \[
        X_{e_1,e_2} 
            = V \setminus \{c_{m+1}, e_1, e_2\},
    \]
    where $\{e_1 < e_2\}$ is an edge not containing $c_{m+1}$. 
    We may assume $e_1 \le e_2 \le c_{m+1}$. 
    Since $H_2$ has $5m-6$ edges and exactly one of them contains $c_{m+1}$, 
    there are $5m - 7$ such faces.

    \item Facets of size $3m-3$, namely
    \[
        Y_a = V \setminus \{c_{m+1}, a\},
    \]
    where $\{c_{m+1}, a\}$ is an independent set (a non-edge), 
    i.e., $a \notin \{c_m, c_{m+1}\}$. 
    There are $3m - 3$ such faces.
    \end{enumerate}
    
    For each $a \notin \{c_{m}, c_{m+1}\}$, the vertex $a$ gives rise to two unmatched faces $A_{a,e_2}$ of the following types:
    
    \begin{enumerate}[label=\roman*., leftmargin=3em]
        \item When $e_2$ is the next \emph{vertical} vertex above $a$ in the same column,  
        i.e., $e_2 = b_i$ if $a = a_i$, and $e_2 = c_i$ if $a = b_i$,  
        except when $a = c_i$ for some $1 \le i \le m-1$  
        (so $a$ cannot lie in the bottom row).
    
        \item When $e_2$ is the next \emph{horizontal} vertex to the right of $a$,  
        i.e., $e_2 = a_{i+1}$ if $a = a_i$,  
        $e_2 = b_{i+1}$ if $a = b_i$,  
        and $e_2 = c_{i+1}$ if $a = c_i$,  
        except when $a \in \{a_{m-1}, b_{m-1}\}$.
    \end{enumerate}
    
    \noindent
    For such $a \notin \{c_m, c_{m+1}\}$, consider the element matching $M_a$.  
    We obtain
    \[
    \begin{cases}
        X_{a,a+1} \longrightarrow X_{a,a+1} \cup \{a\} = Y_{a+1}, 
            & a \neq c_i,\\[4pt]
        X_{a,a+3} \longrightarrow X_{a,a+3} \cup \{a\} = Y_{a+3}, 
            & a \notin \{a_{m-1},\, b_{m-1},\, c_{m-1}\}.
    \end{cases}
    \]

    Here, $a+1$ denotes the next vertical vertex, and $a+3$ denotes the next horizontal vertex in the next column to the right. 
For the case $a = c_{m-1}$, we use the matching
\[
    X_{c_{m-1}, c_{m}} \longrightarrow 
    X_{c_{m-1}, c_{m}} \cup \{c_{m}\} 
    = Y_{c_{m-1}}.
\]

Every Type II unmatched face is paired with a Type I unmatched face through the element matchings $M_a$ for all 
$a \in V_1$, where 
\( V_1 = V \setminus \{c_m, c_{m+1}\} \). 
The matching $\lno \bigcup\limits_{v \in V_1} M_{v} \rno \cup M_{c_{m+1}}$
is acyclic by \Cref{theorem 2}. 
Consequently, only the Type I unmatched faces remain, producing critical cells of dimension $3m-4$, together with one $0$-cell (coming from the empty face). 
The number of such critical cells is
\[
(5m-7) - (3m-3) = 2m - 4.
\]

By \Cref{Corollary 1}, we obtain
\[
    \Delta_2^t \lno H_2 \rno 
    \simeq 
    \bigvee_{2m-4} \bbS^{3m-5}.
\]

This completes the proof.
\end{proof}

\begin{lemma}\label{thm: homotopy type of total 2-cut complex of H1}
    For $m \geq 2$, the total 2-cut complex of $H_3$ is homotopy equivalent to a wedge of $\lno 2m-4 \rno-$ spheres of dimension $\lno 3m-6 \rno-$, i.e., $$\Delta_2^t \lno H_3 \rno \simeq \bigvee\limits_{2m-4}\bbS^{3m-6}.$$
\end{lemma}

We omit the proof of this result, as it follows the same sequence of element matchings and arguments used in the proof of \Cref{thm: homotopy type of total 2-cut complex of (3×m)(1)}. 
We now state the main result of this section.

\begin{theorem}\label{thm: homotopy of 3-cut complex of (3×n)-gg}
    For $n \geq 2$, the total 3-cut complex of $\threegg$, denoted $\Delta_n^t$, is homotopy equivalent to a wedge of $\binom{2n-2}{2}$ spheres of dimension $\lno 3n-6 \rno$, i.e., $$\Delta_n^t \simeq \bigvee\limits_{\binom{2n-2}{2}}\bbS^{3n-6}.$$
\end{theorem}

We prove this result by induction on $n$. 
Since $\Delta_2^t = \totalthreecut\lno \mathcal{G}_{2\times 3}\rno$, the base case $n=2$ follows from \Cref{thm: homotopy type of total k-cut complex of 2n-gg}. 
Assume that \Cref{thm: homotopy of 3-cut complex of (3×n)-gg} holds for all $n \le m$. 
We will show that it also holds for $n = m+1$. 
Before doing so, we discuss some important results that will be used in the proof of \Cref{thm: homotopy of 3-cut complex of (3×n)-gg} for the case $n = m+1$.

\begin{lemma}\label{claim: homotopy of total 3-cut of 3×m mthreeggp}
Let $\Delta_1'$ denote the deletion of $b_{m+1}$ in $\totalthreecut \lno H_1 \rno$, and let $\Delta_2' 
    = \st_{\Delta_1'}\lno b_m \rno \cap \st_{\Delta_1'}\lno c_m \rno.$
Then the following hold:
\begin{enumerate}
    \item $\Delta_2' = L_1 \cup L_2 \cup L_3$, where
    \begin{align*}
        L_1 &= \left\langle \{b_m, c_m, c_{m+1}\} \right\rangle 
              * \Delta_2^t \lno H_3 \rno,\\
        L_2 &= \left\langle \{b_{m-1}, c_{m-1}, a_m\} \right\rangle 
              * \Delta_1^t \lno H_4 \rno,\\
        L_3 &= \left\langle V\lno H_4 \rno \cup \{a_m\} \right\rangle;
    \end{align*}

    \item $L_2 \cap L_3 \subseteq L_1$.
\end{enumerate}
\end{lemma}

\begin{proof}
\begin{enumerate}
    \item We first show that 
    \(
        \Delta_2' \supseteq \bigcup_{i=1}^3 L_i.
    \)
    Let $\sigma \in \bigcup_{i=1}^3 L_i$. 
    If $\sigma \in L_1$, the conclusion is immediate. 
    If $\sigma \in L_2$ or $\sigma \in L_3$, then 
    $\sigma \cup \{b_m\},\, \sigma \cup \{c_m\} \in \Delta_1'$, 
    so $\sigma \in \Delta_2'$.

    Now let $\sigma \in \Delta_2'$, and assume without loss of generality that 
    $\sigma$ is a facet of $\Delta_2'$. 
    Then either 
    $\{b_m, c_m\} \subseteq \sigma$
    or 
    $\{b_m, c_m\} \not\subseteq \sigma$. 
    We consider these cases separately.

    \begin{enumerate}[label=\textbf{\Alph*.}, leftmargin=2.5em]
        \item Suppose $\{b_m, c_m\} \subseteq \sigma$, and let 
        \( S = \{x, y, z\} \) 
        be a $3$-element independent set contained in $\sigma^c$. 
        Clearly $b_{m+1} \in S$; otherwise we have two subcases:

        \begin{enumerate}[label=\roman*., leftmargin=2em]
            \item $c_{m+1} \in S$, say $x = c_{m+1}$;
            \item $c_{m+1} \notin S$.
        \end{enumerate}

        In either situation, define
        \[
            S_1 = S \cup \{b_{m+1}\} \setminus \{x\}.
        \]
        Since $\sigma \subseteq S_1^{\,c}$, we have 
        $\sigma \cup \{x\} \in \Delta_2'$, 
        contradicting the fact that $\sigma$ is a facet of $\Delta_2'$. 
        Hence $b_{m+1} \in S$, which forces 
        $c_{m+1} \in \sigma$, 
        so $\sigma \in L_1$.
        \item Suppose $\{b_m, c_m\} \not\subseteq \sigma$. 
        Assume $b_m \in \sigma$ and $c_m \notin \sigma$ 
        (the reverse situation is handled identically). 
        Since $\sigma \in \st_{\Delta_1'}\lno c_m \rno$ (because $\sigma \in \Delta_2'$), 
        we have 
        \(
            \sigma \cup \{c_m\} \in \Delta_1'.
        \)
        This implies 
        \(
            \sigma \cup \{c_m\} \in \Delta_2',
        \)
        because $b_m \in \sigma \cup \{c_m\}$; 
        hence $\sigma$ cannot be a facet of $\Delta_2'$, a contradiction. 

        By a similar argument, the case 
        $b_m \notin \sigma$ and $c_m \in \sigma$ 
        also leads to a contradiction. 
        Therefore, the only remaining possibility when 
        $\{b_m, c_m\} \not\subseteq \sigma$ 
        is that 
        \(
            \sigma \cap \{b_m, c_m\} = \emptyset.
        \)

    We now claim the following.

    \begin{claim}\label{claim: c_(m+1) in sigma iff sigma in L1}
        If $\sigma$ is a facet of $\Delta_2'$, then 
        $c_{m+1} \in \sigma$ if and only if $\sigma \in L_1$.
    \end{claim}

    \begin{proof}[Proof of \Cref{claim: c_(m+1) in sigma iff sigma in L1}]\renewcommand{\qedsymbol}{}
        The direction $\sigma \in L_1 \Rightarrow c_{m+1} \in \sigma$ follows directly from the definition of $L_1$. 
        For the reverse implication, we argue via the contrapositive:  
        if $\sigma \notin L_1$, then $c_{m+1} \in \sigma^c$.

        Suppose, for contradiction, that $\sigma \notin L_1$ and yet $c_{m+1} \in \sigma$. 
        Since $\sigma \in \st_{\Delta_1'}\lno b_m \rno$, the complement 
        $\bigl( \sigma \cup \{b_m\} \bigr)^c$ 
        contains a $3$-element independent set 
        \( S_2 = \{x_1, x_2, x_3\} \).  
        Because $\sigma$ is a facet of $\Delta_2'$, we must have 
        \( b_{m+1} \in S_2 \), so write \( x_3 = b_{m+1} \).

        If $c_m \notin S_2$, then 
        \(
            \bigl( \sigma \cup \{b_m, c_m\} \bigr)^c
        \)
        still contains $S_2$, implying 
        \( \sigma \cup \{b_m, c_m\} \in \Delta_2' \), 
        contradicting the maximality of $\sigma$.  
        Hence $c_m \in S_2$, say \( x_2 = c_m \), so
        \[
            S_2 = \{x_1,\, c_m,\, b_{m+1}\}.
        \]

    Now, since $\sigma \in \st_{\Delta_1'}\lno c_m \rno$, the complement 
    $\bigl( \sigma \cup \{c_m\} \bigr)^c$ contains a $3$-element independent set, say 
    \( S_3 = \{y_1, y_2, y_3\} \).  
    Because $\sigma \notin L_1$, we must have $b_m \in S_3$, so let $y_3 = b_m$.  
    Together with the assumption that $c_{m+1} \in \sigma$, this forces
    \[
        \{y_1, y_2\} \cap \{b_m, c_m, b_{m+1}, c_{m+1}\} = \emptyset.
    \]
    In this situation, the set 
    \( \{y_1, y_2, b_{m+1}\} \) 
    is a $3$-element independent set in 
    \( \bigl( \sigma \cup \{b_m, c_m\} \bigr)^c \).  
    Hence 
    \( \sigma \cup \{b_m, c_m\} \in \Delta_2' \), 
    contradicting the assumption that $\sigma$ is a facet of $\Delta_2'$.

    This completes the proof of 
    \Cref{claim: c_(m+1) in sigma iff sigma in L1}.
\end{proof}

    Now let $\sigma \in \Delta_2'$ with $\sigma \notin L_1$.  
    The complement 
    \(
        (\sigma \cup \{b_m\})^{c}
    \) 
    contains the vertices $\{c_m, b_{m+1}\}$, and  
    \(
        (\sigma \cup \{c_m\})^{c}
    \) 
    contains the vertices $\{b_m, c_{m+1}\}$.  
    It is easy to see that these two sets require at most two additional vertices, say $x$ and $y$, to form a $3$-element independent set,     \(
        \{x, c_m, b_{m+1}\}
    \) 
    and  
    \(
        \{y, b_m, c_{m+1}\}.
    \)

    One checks that $x = y$ works exactly when 
    \( x \in V\lno H_4 \rno \);  
    in this situation,  
    \( \sigma \in L_2 \).  
    If $x = y$ does not work, then  
    \( \{x, y\} \subseteq \{a_m, b_{m-1}, c_{m-1}\} \),  
    and the vertices $x$ and $y$ must be adjacent because  
    \( \sigma \notin L_1 \).  
    The only such adjacent pair is  
    \( (b_{m-1}, c_{m-1}) \),  
    in which case  
    \( \sigma \in L_3 \).
\end{enumerate}

    \item It is easy to verify that
    \[
        L_2 \cap L_3 
            = \Delta_1^t \lno H_4 \rno * \left\langle \{a_m\} \right\rangle
            \subseteq \Delta_2^t \lno H_3 \rno 
            \subseteq L_1.
    \]
\end{enumerate}
This completes the proof of 
\Cref{claim: homotopy of total 3-cut of 3×m mthreeggp}.
\end{proof}

\begin{lemma}\label{lemma: homotopy of total 3-cut of (3×n)gg'}
    For $m\geq 2$, let $\Delta_m^t \simeq \bigvee\limits_{\binom{2m-2}{2}} \bbS^{3m-6}$. Then, the total 3-cut complex of $H_1$ is homotopy equivalent to a wedge of $\binom{2m-1}{2}$ spheres of dimension $\lno 3m-4 \rno$, i.e., $$\totalthreecut \lno H_1 \rno \simeq \bigvee\limits_{\binom{2m-1}{2}}\bbS^{3m-4}.$$
\end{lemma}

\begin{proof}
    We begin by computing the homotopy types of the link and the deletion of the vertex $b_{m+1}$ in $\totalthreecut \lno H_1 \rno$.

    \medskip

    \noindent\textbf{Link of $b_{m+1}$ in $\totalthreecut \lno H_1 \rno$:}
    Let $H_5 = \mthreeggn \lsq  V \lno \mthreeggn \rno \setminus \lcu a_{m+1},b_{m+1}\rcu \rsq$ (see \Cref{fig: graph of H5}).

    \begin{figure}[h!]
        \centering
        \begin{tikzpicture}[scale=0.9, line cap=round,line join=round,>=triangle 45,x=0.4cm,y=0.4cm]
            \clip(-1,-2) rectangle (18.5,7);
            \draw [line width=1pt] (0,0)-- (0,3);
            \draw [line width=1pt] (0,3)-- (0,6);
            \draw [line width=1pt] (3,0)-- (3,3);
            \draw [line width=1pt] (3,3)-- (3,6);
            \draw [line width=1pt] (6,0)-- (6,3);
            \draw [line width=1pt] (6,3)-- (6,6);
            \draw [line width=1pt] (6,3)-- (7,3);
            \draw [line width=1pt] (6,6)-- (7,6);
            \draw [line width=1pt] (6,0)-- (7,0);
            \draw [line width=1pt] (0,3)-- (6,3);
            \draw [line width=1pt] (0,0)-- (6,0);
            \draw [line width=1pt] (0,6)-- (6,6);
            \draw [line width=1pt,dash pattern=on 4pt off 4pt] (7,6)-- (9,6);
            \draw [line width=1pt,dash pattern=on 4pt off 4pt] (7,3)-- (9,3);
            \draw [line width=1pt,dash pattern=on 4pt off 4pt] (7,0)-- (9,0);
            \draw [line width=1pt] (9,3)-- (10,3);
            \draw [line width=1pt] (9,6)-- (10,6);
            \draw [line width=1pt] (9,0)-- (10,0);
            \draw [line width=1pt] (10,0)-- (10,3);
            \draw [line width=1pt] (10,3)-- (10,6);
            \draw [line width=1pt] (13,0)-- (13,3);
            \draw [line width=1pt] (13,3)-- (13,6);
            \draw [line width=1pt] (10,6)-- (13,6);
            \draw (-0.3,7.3) node[anchor=north west] {\large$a_{_1}$};
            \draw (12.8,7.2) node[anchor=north west] {\large$a_{_{m}}$};
            \draw (9.7,7.35) node[anchor=north west] {\large$a_{_{m-1}}$};
            \draw (5.7,7.3) node[anchor=north west] {\large$a_{_3}$};
            \draw (2.7,7.3) node[anchor=north west] {\large$a_{_2}$};
            \draw (-0.2,4.6) node[anchor=north west] {\large$b_{_1}$};
            \draw (12.8,4.5) node[anchor=north west] {\large$b_{_{m}}$};
            \draw (9.7,4.65) node[anchor=north west] {\large$b_{_{m-1}}$};
            \draw (5.8,4.6) node[anchor=north west] {\large$b_{_3}$};
            \draw (2.8,4.6) node[anchor=north west] {\large$b_{_2}$};
            \draw (-0.2,1.3) node[anchor=north west] {\large$c_{_1}$};
            \draw (12.8,1.35) node[anchor=north west] {\large$c_{_{m}}$};
            \draw (15.8,1.3) node[anchor=north west] {\large$c_{_{m+1}}$};
            \draw (9.8,1.35) node[anchor=north west] {\large$c_{_{m-1}}$};
            \draw (5.8,1.3) node[anchor=north west] {\large$c_{_3}$};
            \draw (2.8,1.3) node[anchor=north west] {\large$c_{_2}$};
            \draw [line width=1pt] (10,3)-- (13,3);
            \draw [line width=1pt] (10,0)-- (13,0);
            
            \draw [line width=1pt] (13,0)-- (16,0);
            \begin{scriptsize}
                \draw [fill=wwwwww] (0,0) circle (2pt);
                \draw [fill=wwwwww] (0,3) circle (2pt);
                \draw [fill=wwwwww] (0,6) circle (2pt);
                \draw [fill=wwwwww] (3,0) circle (2pt);
                \draw [fill=wwwwww] (3,3) circle (2pt);
                \draw [fill=wwwwww] (3,6) circle (2pt);
                \draw [fill=wwwwww] (6,0) circle (2pt);
                \draw [fill=wwwwww] (6,3) circle (2pt);
                \draw [fill=wwwwww] (6,6) circle (2pt);
                \draw [fill=wwwwww] (10,3) circle (2pt);
                \draw [fill=wwwwww] (10,6) circle (2pt);
                \draw [fill=wwwwww] (10,0) circle (2pt);
                \draw [fill=wwwwww] (13,0) circle (2pt);
                \draw [fill=wwwwww] (16,0) circle (2pt);
                \draw [fill=wwwwww] (13,3) circle (2pt);
                \draw [fill=wwwwww] (13,6) circle (2pt);
            \end{scriptsize}
        \end{tikzpicture}
        \vspace{-2em}
        \caption{The graph $H_5$.}
        \label{fig: graph of H5}
    \end{figure}

        Using \Cref{coro:link of a vertex} and \Cref{thm: suspension of previous graph}, we obtain
    \begin{align*}
        \lk_{\totalthreecut \lno H_1 \rno}\lno b_{m+1} \rno 
            &= \Delta_3^t \lno H_5 \rno \\
            &= \Sigma \bigl( \totalthreecut\lno \mthreegg \rno \bigr).
    \end{align*}

    Since 
    \(
        \totalthreecut \lno \mthreegg \rno 
        \simeq 
        \bigvee_{\binom{2m-2}{2}} \bbS^{3m-6},
    \)
    we conclude that
    \begin{equation}\label{link of b_(m+1)}
        \lk_{\totalthreecut \lno H_1 \rno}\lno b_{m+1} \rno
        \simeq 
        \bigvee_{\binom{2m-2}{2}} \bbS^{3m-5}.
    \end{equation}

\noindent\textbf{Deletion of $b_{m+1}$ in $\totalthreecut \lno H_1 \rno$:}
Let $\Delta_1'$ denote the deletion of $b_{m+1}$ in $\totalthreecut \lno H_1 \rno$. 
Using arguments analogous to those in the proof of 
\Cref{lemma: del(b_n) = st(a_(n-1)) U  st(b_(n-1))}, we obtain
\[
    \Delta_1' 
        = \st_{\Delta_1'}\lno b_m \rno 
          \cup 
          \st_{\Delta_1'}\lno c_m \rno.
\]

By \Cref{lemma: bjorner's lemma 10.4(b)}, this gives
\[
    \Delta_1' 
        \simeq 
        \Sigma \bigl( 
            \st_{\Delta_1'}\lno b_m \rno 
            \cap 
            \st_{\Delta_1'}\lno c_m \rno 
        \bigr).
\]

Thus, it remains to compute the homotopy type of  
\(
    \st_{\Delta_1'}\lno b_m \rno 
    \cap 
    \st_{\Delta_1'}\lno c_m \rno.
\)

Let 
\[
    \Delta_2' 
        = \st_{\Delta_1'}\lno b_m \rno 
          \cap 
          \st_{\Delta_1'}\lno c_m \rno.
\]

The complexes $L_1, L_2,$ and \(L_3\) defined in 
\Cref{claim: homotopy of total 3-cut of 3×m mthreeggp} 
are each contractible.  
Hence, by applying 
\Cref{claim: homotopy of total 3-cut of 3×m mthreeggp} together with 
\Cref{lemma: bjorner's lemma 10.4(b)}, we obtain
\[
    \Delta_2' 
        \simeq 
        \Sigma(L_1 \cap L_2)\ 
        \vee\ 
        \Sigma(L_1 \cap L_3).
\]

To determine the homotopy type of $\Delta_2'$, it therefore suffices to compute the homotopy types of 
\(L_1 \cap L_2\) and \(L_1 \cap L_3\).  
We note that
\begin{align*}
    L_1 \cap L_2 
        &= 
        \Delta_2^t \lno H_3 \rno 
           \cap 
           \Bigl(
               \left\langle \{b_{m-1}, c_{m-1}, a_m\} \right\rangle 
               * 
               \Delta_1^t \lno H_4 \rno
           \Bigr) \\
        &= 
        \Delta_2^t \lno H_3 \rno \setminus A_1.
\end{align*}

Here, $A_1$ consists of all subsets of 
$\{a_m, b_{m-1}, c_{m-1}\}$ 
that contain a $2$-element independent set and whose complement is a face of 
$\Delta_2^t \lno H_3 \rno$, namely
\[
    A_1 
    = 
    \bigl\{
        \{b_{m-1}, a_m\},\ 
        \{c_{m-1}, a_m\},\ 
        \{b_{m-1}, c_{m-1}, a_m\}
    \bigr\}.
\]

Thus, determining the homotopy type of 
$L_1 \cap L_2$ reduces to determining the homotopy type of 
$\Delta_2^t \lno H_3 \rno \setminus A_1$.  
From \Cref{thm: homotopy type of total 2-cut complex of H1}, we have
\[
    \Delta_2^t \lno H_3 \rno 
        \simeq 
        \bigvee_{2m-4} \bbS^{3m-6}.
\]

Recall that the proof of 
\Cref{thm: homotopy type of total 2-cut complex of H1} 
(similar to the proof of 
\Cref{thm: homotopy type of total 2-cut complex of (3×m)(1)}) 
was carried out using a Morse matching.  
To compute the homotopy type of 
$\Delta_2^t \lno H_3 \rno \setminus A_1$, 
we use the same matching scheme, except that we now remove those faces of 
$\Delta_2^t \lno H_3 \rno$ whose complements lie in $A_1$.

Following the sequence of element matchings described in 
\Cref{thm: homotopy type of total 2-cut complex of H1}, 
the face $\{b_{m-1}, c_{m-1}, a_m\}^c$ is matched with 
$\{c_{m-1}, a_m\}^c$ via the element matching $M_{b_{m-1}}$.  
Similarly, the face $\{b_{m-1}, a_m\}^c$ is matched with 
$\{a_{m-1}, b_{m-1}, a_m\}^c$ via the element matching $M_{a_{m-1}}$.  

In the complex $\Delta_2^t \lno H_3 \rno \setminus A_1$, 
the face $\{b_{m-1}, a_m\}^c$ is removed, and therefore 
$\{a_{m-1}, b_{m-1}, a_m\}^c$ remains unmatched.  
This contributes one additional unmatched face.  
Hence, the total number of unmatched faces is 
\[
    (2m-4) + 1 = 2m-3,
\]
each of cardinality $3m-5$.  
Thus,
\[
    L_1 \cap L_2 
        \simeq 
        \bigvee_{2m-3} \bbS^{3m-6}.
\]

Next, consider $L_1 \cap L_3$.  We observe that
\begin{align*}
    L_1 \cap L_3 
        &= \Delta_2^t \lno H_3 \rno 
           \cap 
           \left\langle V\lno H_4 \rno \cup \{a_m\} \right\rangle \\
        &= \Delta_1^t \lno H_6 \rno,
\end{align*}
where 
\[
    H_6 
    = \mthreegg \!\left[ 
        V\lno \mthreegg \rno 
        \setminus \{b_{m-1}, c_{m-1}, b_m, c_m\}
      \right]
\]
(see \Cref{fig: graph of H3}).  
Thus, computing the homotopy type of $L_1 \cap L_3$ reduces to computing the homotopy type of 
$\Delta_1^t \lno H_6 \rno$, which is precisely the boundary of a simplex on the vertex set 
$V\lno H_6 \rno$.  
Hence,
\[
    L_1 \cap L_3 \simeq \bbS^{3m-6}.
\]

        \begin{figure}[H]
    \centering
    \begin{tikzpicture}[line cap=round,line join=round,>=triangle 45,x=0.4cm,y=0.4cm]
            \clip(-1,-2) rectangle (20.5,7);
            \draw [line width=1pt] (0,0)-- (0,3);
            \draw [line width=1pt] (0,3)-- (0,6);
            \draw [line width=1pt] (3,0)-- (3,3);
            \draw [line width=1pt] (3,3)-- (3,6);
            \draw [line width=1pt] (6,0)-- (6,3);
            \draw [line width=1pt] (6,3)-- (6,6);
            \draw [line width=1pt] (6,3)-- (7,3);
            \draw [line width=1pt] (6,6)-- (7,6);
            \draw [line width=1pt] (6,0)-- (7,0);
            \draw [line width=1pt] (0,3)-- (6,3);
            \draw [line width=1pt] (0,0)-- (6,0);
            \draw [line width=1pt] (0,6)-- (6,6);
            \draw [line width=1pt,dash pattern=on 4pt off 4pt] (7,6)-- (9,6);
            \draw [line width=1pt,dash pattern=on 4pt off 4pt] (7,3)-- (9,3);
            \draw [line width=1pt,dash pattern=on 4pt off 4pt] (7,0)-- (9,0);
            \draw [line width=1pt] (9,3)-- (10,3);
            \draw [line width=1pt] (9,6)-- (10,6);
            \draw [line width=1pt] (9,0)-- (10,0);
            \draw [line width=1pt] (10,0)-- (10,3);
            \draw [line width=1pt] (10,3)-- (10,6);
            \draw [line width=1pt] (13,0)-- (13,3);
            \draw [line width=1pt] (13,3)-- (13,6);
            \draw [line width=1pt] (10,6)-- (13,6);
            \draw [line width=1pt] (16,6)-- (19,6);
            
            \draw (-0.3,7.2) node[anchor=north west] {\large$a_{_1}$};
            \draw (12.7,7.2) node[anchor=north west] {\large$a_{_{m-2}}$};
            \draw (9.7,7.2) node[anchor=north west] {\large$a_{_{m-3}}$};
            \draw (5.7,7.2) node[anchor=north west] {\large$a_{_3}$};
            \draw (2.7,7.2) node[anchor=north west] {\large$a_{_2}$};
            \draw (15.8,7.2) node[anchor=north west] {\large$a_{_{m-1}}$};
            \draw (18.8,7.1) node[anchor=north west] {\large$a_{_{m}}$};
            
            \draw (-0.2,4.5) node[anchor=north west] {\large$b_{_1}$};
            \draw (12.8,4.4) node[anchor=north west] {\large$b_{_{m-2}}$};
            \draw (9.8,4.5) node[anchor=north west] {\large$b_{_{m-3}}$};
            \draw (5.8,4.5) node[anchor=north west] {\large$b_{_3}$};
            \draw (2.8,4.5) node[anchor=north west] {\large$b_{_2}$};
            
            \draw (-0.2,1.2) node[anchor=north west] {\large$c_{_1}$};
            \draw (12.8,1.1) node[anchor=north west] {\large$c_{_{m-2}}$};
            \draw (9.8,1.2) node[anchor=north west] {\large$c_{_{m-3}}$};
            \draw (5.8,1.2) node[anchor=north west] {\large$c_{_3}$};
            \draw (2.8,1.2) node[anchor=north west] {\large$c_{_2}$};
            \draw [line width=1pt] (10,3)-- (13,3);
            \draw [line width=1pt] (10,0)-- (13,0);
            \draw [line width=1pt] (13,6)-- (16,6);
            \begin{scriptsize}
                \draw [fill=wwwwww] (0,0) circle (2pt);
                \draw [fill=wwwwww] (0,3) circle (2pt);
                \draw [fill=wwwwww] (0,6) circle (2pt);
                \draw [fill=wwwwww] (3,0) circle (2pt);
                \draw [fill=wwwwww] (3,3) circle (2pt);
                \draw [fill=wwwwww] (3,6) circle (2pt);
                \draw [fill=wwwwww] (6,0) circle (2pt);
                \draw [fill=wwwwww] (6,3) circle (2pt);
                \draw [fill=wwwwww] (6,6) circle (2pt);
                \draw [fill=wwwwww] (10,3) circle (2pt);
                \draw [fill=wwwwww] (10,6) circle (2pt);
                \draw [fill=wwwwww] (10,0) circle (2pt);
                \draw [fill=wwwwww] (13,0) circle (2pt);
                \draw [fill=wwwwww] (16,6) circle (2pt);
                \draw [fill=wwwwww] (13,3) circle (2pt);
                \draw [fill=wwwwww] (13,6) circle (2pt);
                \draw [fill=wwwwww] (19,6) circle (2pt);
            \end{scriptsize}
        \end{tikzpicture}
        \vspace{-2em}
        \caption{The graph $H_6$.}
        \label{fig: graph of H3}
    \end{figure}

Therefore,
\[
    \Delta_2' 
    = 
    \st_{\Delta_1'}\lno b_m \rno 
      \cap 
      \st_{\Delta_1'}\lno c_m \rno
    \simeq 
    \bigvee_{2m-2} \bbS^{3m-5}.
\]
Hence,
\begin{equation}\label{deletion of b_(n+1)}
    \Delta_1' 
    = 
    \del_{\totalthreecut \lno H_1 \rno}\lno b_{m+1} \rno
    \simeq 
    \bigvee_{2m-2} \bbS^{3m-4}.
\end{equation}

From \Cref{link of b_(m+1)}, the link of $b_{m+1}$ is homotopy equivalent to a wedge of spheres of dimension $(3m-5)$, while \Cref{deletion of b_(n+1)} shows that the deletion of $b_{m+1}$ is homotopy equivalent to a wedge of spheres of dimension $(3m-4)$.  
Thus, $\lk_{\totalthreecut \lno H_1 \rno}\lno b_{m+1} \rno$ is contractible inside $\del_{\totalthreecut \lno H_1 \rno}\lno b_{m+1} \rno,$ 
and applying \Cref{lemma: bjorner's lemma 10.4(b)} yields
\begin{align*}
    \totalthreecut \lno H_1 \rno 
        &\simeq 
        \del_{\totalthreecut \lno H_1 \rno}\lno b_{m+1} \rno 
        \vee 
        \Sigma \lno 
            \lk_{\totalthreecut \lno H_1 \rno}\lno b_{m+1} \rno
        \rno \\
        &\simeq 
        \left( \bigvee_{2m-2} \bbS^{3m-4} \right)
        \vee 
        \Sigma \left( \bigvee_{\binom{2m-2}{2}} \bbS^{3m-5} \right) \\
        &\simeq 
        \left( \bigvee_{\binom{2m-2}{1}}\bbS^{3m-4} \right)
        \vee 
        \left( \bigvee_{\binom{2m-2}{2}}\bbS^{3m-4} \right) \\
        &\simeq 
        \bigvee_{\binom{2m-1}{2}} \bbS^{3m-4}.
\end{align*}

This completes the proof of 
\Cref{lemma: homotopy of total 3-cut of (3×n)gg'}.
\end{proof}

Before stating the next lemma, we recall the graph $H_2$ (see \Cref{fig: graph of G(1)(3×n)}). 
We also introduce the graph $H_7$ (see \Cref{fig: graph of H7}), defined by
\[
    H_7 
    = 
    \mathcal{G}_{3 \times (m-1)}
    \bigl[
        V\lno \mathcal{G}_{3 \times (m-1)} \rno 
        \setminus 
        \{a_{m-1}, b_{m-1}\}
    \bigr].
\]

\begin{figure}[h!]
        \centering
        \begin{tikzpicture}[scale=0.9, line cap=round,line join=round,>=triangle 45,x=0.4cm,y=0.4cm]
            \clip(-1,-2) rectangle (18.5,7);
            \draw [line width=1pt] (0,0)-- (0,3);
            \draw [line width=1pt] (0,3)-- (0,6);
            \draw [line width=1pt] (3,0)-- (3,3);
            \draw [line width=1pt] (3,3)-- (3,6);
            \draw [line width=1pt] (6,0)-- (6,3);
            \draw [line width=1pt] (6,3)-- (6,6);
            \draw [line width=1pt] (6,3)-- (7,3);
            \draw [line width=1pt] (6,6)-- (7,6);
            \draw [line width=1pt] (6,0)-- (7,0);
            \draw [line width=1pt] (0,3)-- (6,3);
            \draw [line width=1pt] (0,0)-- (6,0);
            \draw [line width=1pt] (0,6)-- (6,6);
            \draw [line width=1pt,dash pattern=on 4pt off 4pt] (7,6)-- (9,6);
            \draw [line width=1pt,dash pattern=on 4pt off 4pt] (7,3)-- (9,3);
            \draw [line width=1pt,dash pattern=on 4pt off 4pt] (7,0)-- (9,0);
            \draw [line width=1pt] (9,3)-- (10,3);
            \draw [line width=1pt] (9,6)-- (10,6);
            \draw [line width=1pt] (9,0)-- (10,0);
            \draw [line width=1pt] (10,0)-- (10,3);
            \draw [line width=1pt] (10,3)-- (10,6);
            \draw [line width=1pt] (13,0)-- (13,3);
            \draw [line width=1pt] (13,3)-- (13,6);
            \draw [line width=1pt] (10,6)-- (13,6);
            \draw (-0.3,7.3) node[anchor=north west] {\large$a_{_1}$};
            \draw (12.8,7.2) node[anchor=north west] {\large$a_{_{m - 2}}$};
            \draw (9.7,7.35) node[anchor=north west] {\large$a_{_{m-3}}$};
            \draw (5.7,7.3) node[anchor=north west] {\large$a_{_3}$};
            \draw (2.7,7.3) node[anchor=north west] {\large$a_{_2}$};
            \draw (-0.2,4.6) node[anchor=north west] {\large$b_{_1}$};
            \draw (12.8,4.5) node[anchor=north west] {\large$b_{_{m-2}}$};
            \draw (9.75  ,4.65) node[anchor=north west] {\large$b_{_{m-3}}$};
            \draw (5.8,4.6) node[anchor=north west] {\large$b_{_3}$};
            \draw (2.8,4.6) node[anchor=north west] {\large$b_{_2}$};
            \draw (-0.2,1.3) node[anchor=north west] {\large$c_{_1}$};
            \draw (12.8,1.35) node[anchor=north west] {\large$c_{_{m-2}}$};
            \draw (15.8,1.3) node[anchor=north west] {\large$c_{_{m-1}}$};
            \draw (9.8,1.35) node[anchor=north west] {\large$c_{_{m-3}}$};
            \draw (5.8,1.3) node[anchor=north west] {\large$c_{_3}$};
            \draw (2.8,1.3) node[anchor=north west] {\large$c_{_2}$};
            \draw [line width=1pt] (10,3)-- (13,3);
            \draw [line width=1pt] (10,0)-- (13,0);
            
            \draw [line width=1pt] (13,0)-- (16,0);
            \begin{scriptsize}
                \draw [fill=wwwwww] (0,0) circle (2pt);
                \draw [fill=wwwwww] (0,3) circle (2pt);
                \draw [fill=wwwwww] (0,6) circle (2pt);
                \draw [fill=wwwwww] (3,0) circle (2pt);
                \draw [fill=wwwwww] (3,3) circle (2pt);
                \draw [fill=wwwwww] (3,6) circle (2pt);
                \draw [fill=wwwwww] (6,0) circle (2pt);
                \draw [fill=wwwwww] (6,3) circle (2pt);
                \draw [fill=wwwwww] (6,6) circle (2pt);
                \draw [fill=wwwwww] (10,3) circle (2pt);
                \draw [fill=wwwwww] (10,6) circle (2pt);
                \draw [fill=wwwwww] (10,0) circle (2pt);
                \draw [fill=wwwwww] (13,0) circle (2pt);
                \draw [fill=wwwwww] (16,0) circle (2pt);
                \draw [fill=wwwwww] (13,3) circle (2pt);
                \draw [fill=wwwwww] (13,6) circle (2pt);
            \end{scriptsize}
        \end{tikzpicture}
        \vspace{-2em}
        \caption{The graph $H_7$.}
        \label{fig: graph of H7}
    \end{figure}

\begin{lemma}\label{claim: 1 (homotopy of total 3-cut of 3×n gg)}
    Let $\Delta_1$ denote the deletion of $a_{m+1}$ in $\Delta_{m+1}^t$, and $\Delta_2 = \st_{\Delta_1}\lno a_m \rno \cap \st_{\Delta_1}\lno b_m \rno$. Then,
        \begin{enumerate}
            \item  $\Delta_2 = K_1 \cup K_2 \cup K_3 \cup K_4$, where
            \begin{align*}
                K_1 &= \left\langle \lcu a_m, b_m, b_{m+1} \rcu \right\rangle * \Delta_2^t \lno H_2 \rno,\\
                K_2 &= \left\langle \lcu a_{m-1}, b_{m-1}, c_m, c_{m+1} \rcu \right\rangle * \Delta_1^t \lno H_7 \rno,\\
                K_3 &= \left\langle V\lno H_7 \rno \right\rangle * \left\langle \lcu a_{m-1}, b_{m-1} \rcu \right\rangle,\\
                K_4 &= \left\langle V\lno H_7 \rno \right\rangle * \left\langle \lcu c_{m}, c_{m+1} \rcu \right\rangle;
            \end{align*}
             
            \item $K_i \cap K_j \subseteq K_1$, where $2\leq i < j \leq 4$.
        \end{enumerate}
\end{lemma}    
    
\begin{proof}
    \begin{enumerate}
        \item We first show that $\Delta_{2} \supseteq \bigcup\limits_{i=1}^4{K_i}$. 
        Let $\sigma \in \bigcup\limits_{i=1}^4{K_i}$. If $\sigma \in K_1$, then the result is immediate. If $\sigma \in K_i$ with $i\neq1$, then $\sigma \cup \lcu a_m \rcu\in \Delta_1$ and $\sigma \cup \lcu b_m \rcu \in \Delta_1$, hence $\sigma \in \Delta_2$.

        To show that $\Delta_{2} \subseteq \bigcup\limits_{i=1}^4{K_i}$, let $\sigma \in \Delta_2$, and WLOG, assume that $\sigma$ is a facet of $\Delta_2$. Then, either $\lcu a_m, b_m\rcu \subseteq \sigma$ or $\lcu a_m, b_m\rcu \not\subseteq \sigma$. We analyse these two cases separately below.

        \begin{enumerate}[label = \Alph*.]
    \item Let $\lcu a_m, b_m \rcu \subseteq \sigma$, and let $S = \lcu x, y, z \rcu$ be an independent set of size $3$ contained in $\sigma^c$. Clearly $a_{m+1} \in S$. If $a_{m+1} \notin S$, consider the following two sub-cases:
    \begin{enumerate}[label = \roman*.]
        \item $b_{m+1} \in S$, say $x = b_{m+1}$.
        \item $b_{m+1} \notin S$.
    \end{enumerate}

    In both sub-cases, construct $S_1 = S \cup \lcu a_{m+1} \rcu \setminus \lcu x \rcu.$ 
    Since $\sigma \subseteq S_1^c$, it follows that $\sigma \cup \lcu x \rcu \in \Delta_2$, which contradicts the fact that $\sigma$ is a facet of $\Delta_2$. Therefore, if $a_{m+1} \in S$, then $b_{m+1} \in \sigma$, and this implies $\sigma \in K_1$.

    \item Suppose $\lcu a_m, b_m \rcu \not\subseteq \sigma$. Assume $a_m \in \sigma$ and $b_m \notin \sigma$; the argument for the opposite situation is identical. Since $\sigma \in \st_{\Delta_1}\lno b_m\rno$ (because $\sigma \in \Delta_2$), we have $\sigma \cup \lcu b_m \rcu \in \Delta_1$. This would imply $\sigma \cup \lcu b_m \rcu \in \Delta_2$, contradicting that $\sigma$ is a facet of $\Delta_2$. By the same reasoning, we obtain $\sigma \cap \lcu a_m, b_m \rcu = \emptyset$. Hence, the case $\lcu a_m, b_m \rcu \not\subseteq \sigma$ occurs only when $\sigma \cap \lcu a_m, b_m \rcu = \emptyset$.

    We now claim the following.

    \begin{claim}\label{claim: b_(m+1) in sigma iff sigma in K1}
    If $\sigma$ is a facet in $\Delta_2$, then $b_{m+1} \in \sigma$ if and only if $\sigma \in K_1$.
\end{claim}

\begin{proof}[Proof of \Cref{claim: b_(m+1) in sigma iff sigma in K1}]
    The implication $\sigma \in K_1 \Rightarrow b_{m+1} \in \sigma$ follows immediately from the definition of $K_1$. For the converse, we prove the contrapositive: if $\sigma \notin K_1$, then $b_{m+1} \in \sigma^c$.

    Assume toward a contradiction that $\sigma \notin K_1$ and $b_{m+1} \in \sigma$. Since $\sigma \in \st_{\Delta_1}\lno a_m \rno$, the complement $\lno \sigma \cup \lcu a_m \rcu \rno^c$ contains an independent set of size $3$, say $S_2 = \lcu x_1, x_2, x_3 \rcu$. Because $\sigma$ is a facet of $\Delta_2$, we must have $a_{m+1} \in S_2$. Let $x_3 = a_{m+1}$.

    If $b_m \notin S_2$, then $\lno \sigma \cup \lcu a_m, b_m \rcu \rno^c$ also contains $S_2$ as an independent set, which would imply $\sigma \cup \lcu a_m, b_m \rcu \in \Delta_2$, contradicting the fact that $\sigma$ is a facet of $\Delta_2$. Therefore $b_m \in S_2$, and we may write $ S_2 = \lcu x_1, b_m, a_{m+1} \rcu.$

    Now, using $\sigma \in \st_{\Delta_1}\lno b_m \rno$, we obtain that $\lno \sigma \cup \lcu b_m \rcu \rno^c$ contains an independent set of size $3$, say $S_3 = \lcu y_1, y_2, y_3 \rcu$. Since $\sigma \notin K_1$, we have $a_m \in S_3$. Let $y_3 = a_m$. Together with the assumption that $b_{m+1} \in \sigma$, this implies
    $\{y_1, y_2\} \cap \{a_m, b_m, a_{m+1}, b_{m+1}\} = \emptyset.$    
    In this situation, the set $\{y_1, y_2, a_{m+1}\}$ forms a size-$3$ independent set in $\lno \sigma \cup \lcu a_m, b_m \rcu \rno^c$, and therefore
    $ \sigma \cup \lcu a_m, b_m \rcu \in \Delta_2,$
    which contradicts the assumption that $\sigma$ is a facet of $\Delta_2$.
    
    This completes the proof of \Cref{claim: b_(m+1) in sigma iff sigma in K1}.
    \end{proof}

    Now, let $\sigma \in \Delta_2$ with $\sigma \notin K_1$. Then 
    \[
        \lno \sigma \cup \lcu a_m \rcu \rno^c \text{ contains } \lcu b_m, a_{m+1} \rcu
        \quad\text{and}\quad
        \lno \sigma \cup \lcu b_m \rcu \rno^c \text{ contains } \lcu a_m, b_{m+1} \rcu
    \]
    in every independent set of size $3$. Thus they require at most two additional vertices, say $x$ and $y$, to form such an independent set, giving the sets $\lcu x, b_m, a_{m+1} \rcu$ and $\lcu y, a_m, b_{m+1} \rcu$.
    
    It is easy to verify that $x = y$ works if and only if $x \in V\lno H_7 \rno$, in which case $\sigma \in K_2$. If $x = y$ does not work, then 
    \[
        \lcu x, y \rcu \subseteq \lcu a_{m-1}, b_{m-1}, c_m, c_{m+1} \rcu,
    \]
    with $x$ and $y$ adjacent because $\sigma \notin K_1$. Hence the pair $\lno x, y \rno$ must be one of the following:
    \begin{enumerate}[label = \roman*.]
        \item $\lno x, y \rno = \lno c_m, c_{m+1} \rno$, in which case $\sigma \in K_3$;
        \item $\lno x, y \rno = \lno a_{m-1}, b_{m-1} \rno$, in which case $\sigma \in K_4$.
    \end{enumerate}
    \end{enumerate}

Thus, if $\sigma \in \Delta_2$, then $\sigma \in \bigcup\limits_{i=1}^4 K_i$. This proves \Cref{claim: 1 (homotopy of total 3-cut of 3×n gg)}(1).

\item It is straightforward to check that
\begin{align*}
    K_2 \cap K_3 &= \left\langle \lcu a_{m-1}, b_{m-1} \rcu \right\rangle * \Delta_1^t \lno H_7 \rno \subseteq \Delta_2^t \lno H_2 \rno,\\
    K_2 \cap K_4 &= \left\langle \lcu c_{m}, c_{m+1} \rcu \right\rangle * \Delta_1^t \lno H_7 \rno \subseteq \Delta_2^t \lno H_2 \rno,\\
    K_3 \cap K_4 &= \left\langle V\lno H_7 \rno \right\rangle \subseteq \Delta_2^t \lno H_2 \rno.
\end{align*}
In each case above, we have $K_i \cap K_j \subseteq K_1$ for every $2 \leq i < j \leq 4$.
\end{enumerate}

This completes the proof of \Cref{claim: 1 (homotopy of total 3-cut of 3×n gg)}.
\end{proof}

We are now ready to prove \Cref{thm: homotopy of 3-cut complex of (3×n)-gg}. Following steps similar to those in the proof of \Cref{thm: homotopy type of total k-cut complex of 2n-gg}, to establish the result for $n = m+1$ we first compute the homotopy types of the link and deletion of the vertex $a_{m+1}$ in $\Delta_{m+1}^t$ (see \Cref{link of a_(m+1)} and \Cref{deletion of a_(m+1)} respectively). We then observe that the link of $a_{m+1}$ is contractible inside its deletion in $\Delta_{m+1}^t$, and therefore, by \Cref{lemma: bjorner's lemma 10.4(b)}, the desired result follows.

\begin{proof}[Proof of \Cref{thm: homotopy of 3-cut complex of (3×n)-gg}]
Recall that we are proving this result by induction. Assuming that the statement holds for all $n \leq m$, we aim to prove it for $n = m+1$.

We begin by computing the homotopy types of the link and deletion of the vertex $a_{m+1}$ in $\Delta_{m+1}^t$.

\medskip
\noindent\textbf{Link of the vertex $a_{m+1}$ in $\Delta_{m+1}^t$:}

Using \Cref{coro:link of a vertex} and \Cref{lemma: homotopy of total 3-cut of (3×n)gg'}, we obtain
\begin{align}\label{link of a_(m+1)}
    \lk_{\Delta_{m+1}^t}\lno a_{m+1} \rno
    = \totalthreecut \lno H_1 \rno
    \simeq \bigvee\limits_{\binom{2m-1}{2}} \bbS^{3m-4}.
\end{align}

\noindent\textbf{Deletion of the vertex $a_{m+1}$ in $\Delta_{m+1}^t$:}

Recall that $\Delta_1 = \del_{\Delta_{m+1}^t}\lno a_{m+1} \rno$. Using \Cref{lemma: del(a_n) = st(a_(n-1)) U  st(b_(n-1))}, we have
\[
    \Delta_1 = \st_{\Delta_1}\lno a_m \rno \cup \st_{\Delta_1}\lno b_m \rno.
\]

By \Cref{lemma: bjorner's lemma 10.4(b)}, it follows that
\[
    \Delta_1 \simeq \Sigma \lno \st_{\Delta_1}\lno a_m \rno \cap \st_{\Delta_1}\lno b_m \rno \rno.
\]
Hence, it suffices to compute the homotopy type of the intersection $\st_{\Delta_1}\lno a_m \rno \cap \st_{\Delta_1}\lno b_m \rno$.

Recall that $\Delta_2 = \st_{\Delta_1}\lno a_m \rno \cap \st_{\Delta_1}\lno b_m \rno$. Observe that the sets $K_i$ defined in \Cref{claim: 1 (homotopy of total 3-cut of 3×n gg)} are all contractible. Therefore, by \Cref{claim: 1 (homotopy of total 3-cut of 3×n gg)} together with \Cref{lemma: bjorner's lemma 10.4(b)}, we obtain
\[
    \Delta_2 \simeq 
    \bigvee\limits_{2 \leq i \leq 4}
    \lno \Sigma \lno K_1 \cap K_i \rno \rno.
\]

Therefore, to determine the homotopy type of $\Delta_2$, it is enough to compute the homotopy type of $K_1 \cap K_j$ for $j \neq 1$. Observe that
\begin{align*}
    K_1 \cap K_2
        &= \Delta_2^t \lno H_2 \rno
           \cap 
           \left\langle 
               \left\langle 
                   \lcu a_{m-1}, b_{m-1}, c_m, c_{m+1} \rcu 
               \right\rangle 
               * \Delta_1^t \lno H_7 \rno
           \right\rangle \\
        &= \Delta_2^t \lno H_2 \rno \setminus A_2.
\end{align*}
Here, $A_2$ consists of all subsets of $ \lcu a_{m-1}, b_{m-1}, c_m, c_{m+1} \rcu $ that contain a size two independent set and whose complement is a face of $\Delta_2^t \lno H_2 \rno$. Explicitly,
\begin{align*}
    A_2 = \Big\{
        \lcu b_{m-1}, c_m \rcu,\;
         \lcu a_{m-1}, c_m \rcu,\;
         \lcu a_{m-1}, c_{m+1} \rcu,\;
         \lcu b_{m-1}, c_{m+1} \rcu,\;
        \lcu a_{m-1}, b_{m-1}, c_m \rcu,\\
         \lcu a_{m-1}, b_{m-1}, c_{m+1} \rcu,\;
         \lcu a_{m-1}, c_m, c_{m+1} \rcu,\;
        \lcu b_{m-1}, c_m, c_{m+1} \rcu,\;
         \lcu a_{m-1}, b_{m-1}, c_m, c_{m+1} \rcu
    \Big\}.
\end{align*}

Thus, to determine the homotopy type of $K_1 \cap K_2$, it suffices to compute the homotopy type of $\Delta_2^t \lno H_2 \rno \setminus A_2$. From \Cref{thm: homotopy type of total 2-cut complex of (3×m)(1)}, we have
\[
    \Delta_2^t \lno H_2 \rno \simeq \bigvee\limits_{2m-4} \bbS^{3m-5}.
\]

Recall that the proof of \Cref{thm: homotopy type of total 2-cut complex of (3×m)(1)} was carried out using a Morse matching. In determining the homotopy type of $\Delta_2^t \lno H_2 \rno \setminus A_2$, we will use the same approach, except that we now exclude those faces of $\Delta_2^t \lno H_2 \rno$ whose complements lie in $A_2$.

After carefully looking at the sequence of element matchings in the proof of \Cref{thm: homotopy type of total 2-cut complex of (3×m)(1)}, we make the following observations:
\begin{itemize}
    \item The faces corresponding to $\lcu a_{m-1},c_{m},c_{m+1}\rcu^c$, $\lcu b_{m-1},c_{m},c_{m+1}\rcu^c$, $\lcu a_{m-1},b_{m-1},c_{m},c_{m+1}\rcu^c$ are matched with the faces $\lcu a_{m-1},c_{m}\rcu^c$, $\lcu b_{m-1},c_{m}\rcu^c$, $\lcu a_{m-1},b_{m-1},c_{m+1}\rcu^c$ respectively, through the element matching $M_{c_{m+1}}$.

    \item The face corresponding to $\lcu a_{m-1}, b_{m-1}, c_{m+1} \rcu^{c}$ is matched with $\lcu b_{m-1}, c_{m+1} \rcu^{c}$.
\end{itemize}

Moreover, note that the face corresponding to $\lcu a_{m-1}, c_{m+1} \rcu^{c}$ would normally be matched with the face corresponding to $\lcu a_{m-2}, a_{m-1}, c_{m+1} \rcu^{c}$ via the element matching $M_{a_{m-2}}$. Since the face $\lcu a_{m-1}, c_{m+1} \rcu^{c}$ is excluded, the face $\lcu a_{m-2}, a_{m-1}, c_{m+1} \rcu^{c}$ remains unmatched, contributing one additional critical cell. Thus, the total number of unmatched faces becomes $2m-4 + 1 = 2m-3$, and therefore
\[
    K_1 \cap K_2 \simeq \bigvee\limits_{2m-3} \bbS^{3m-5}.
\]

To compute $K_1 \cap K_3$, observe that
\begin{align*}
    K_1 \cap K_3
        &= \Delta_2^t \lno H_2 \rno
           \cap
           \left\langle 
               \left\langle V\lno H_7 \rno \right\rangle
               * \left\langle \lcu a_{m-1}, b_{m-1} \rcu \right\rangle
           \right\rangle \\
        &= \Delta_1^t \lno H_2 \cup \lno a_{m-1}, b_{m-1} \rno \rno \\
        &\simeq \bbS^{3m-5}.
\end{align*}

Similarly, to compute the homotopy type of $K_1 \cap K_4$, observe that
\begin{align*}
    K_1 \cap K_4
        &= \Delta_2^t \lno H_2 \rno
           \cap
           \left\langle 
               \left\langle V\lno H_7 \rno \right\rangle
               * \left\langle \lcu c_m, c_{m+1} \rcu \right\rangle
           \right\rangle \\
        &= \Delta_1^t \lno H_2 \cup \lno c_m, c_{m+1} \rno \rno \\
        &\simeq \bbS^{3m-5}.
\end{align*}

Therefore,
\begin{align*}
    \Delta_2
        = \st_{\Delta_1}\lno a_m \rno
          \cap
          \st_{\Delta_1}\lno b_m \rno
        \simeq 
        \bigvee\limits_{2 \leq i \leq 4}
            \lno \Sigma \lno K_1 \cap K_i \rno \rno
        \simeq 
        \bigvee\limits_{\,2m-1}
            \bbS^{3m-4}.
\end{align*}
Furthermore,
\begin{align}\label{deletion of a_(m+1)}
    \Delta_1 = \del_{\Delta_{m+1}^t}\lno a_{m+1} \rno
    \simeq \Sigma(\Delta_2)
    \simeq \bigvee\limits_{2m-1} \bbS^{3m-3}.
\end{align}

From \Cref{link of a_(m+1)}, the link of $a_{m+1}$ in $\Delta_{m+1}^t$ is homotopy equivalent to a wedge of spheres of dimension $3m-4$, and \Cref{deletion of a_(m+1)} shows that the deletion of $a_{m+1}$ in $\Delta_{m+1}^t$ is homotopy equivalent to a wedge of spheres of dimension $3m-3$. Hence, $\lk_{\Delta_{m+1}^t}\lno a_{m+1} \rno$ is contractible inside $\del_{\Delta_{m+1}^t}\lno a_{m+1} \rno$, and therefore, by \Cref{lemma: bjorner's lemma 10.4(b)}, we obtain
\begin{align*}
    \Delta_{m+1}^t
        &\simeq 
        \Bigl( \bigvee\limits_{2m-1} \bbS^{3m-3} \Bigr)
        \vee
        \Sigma \Bigl( \bigvee\limits_{\binom{2m-1}{2}} \bbS^{3m-4} \Bigr) \simeq 
        \bigvee\limits_{\binom{2m-1}{1} + \binom{2m-1}{2}}
            \bbS^{3m-3} \simeq 
        \bigvee\limits_{\binom{2m}{2}}
            \bbS^{3m-3}.
\end{align*}

This completes the proof of \Cref{thm: homotopy of 3-cut complex of (3×n)-gg}.
\end{proof}

\section{$k$-cut complex of $\left( 2 \times n \right) $ grid graph}\label{section: shellability of k-cut of 2n-gg}

In this section, we prove that the $k$-cut complex of $\twogg$ is shellable for all $n \geq 3$ and $3 \leq k \leq 2n-2$. Note that the upper bound $k = 2n-2$ is sharp, since for $k \geq 2n-1$ the complex $\kcut\lno \twogg \rno$ is a void complex. Recall the ordering used in the proofs of \Cref{thm: homotopy type of total 2-cut complex of (3×m)(1)} and \Cref{thm: homotopy type of total 2-cut complex of H1}. We again use the same ordering on the vertices of $\twogg$ (see \Cref{fig: graph of G(2×n)}):
\begin{itemize}
    \item If $i < j$, then $a_i < a_j$, $b_i < b_j$, $a_i < b_j$, and $b_i < a_j$.
    \item If $i = j$, then $a_i < b_j$.
\end{itemize}

Before discussing the shellability of $\kcut \lno \twogg \rno$, we first establish several results concerning the $k$-cut complexes of $\twogg$, $\twoggp$, and certain vertex deletions in these complexes. See \Cref{fig: graph of G'(2×n)} for the definition of the graph $\twoggp$.

\begin{lemma}\label{lemma: 2n is shedding vertex of augumented grid graph}
    For $\kcut \lno \twoggp \rno$, the vertex $a_{n+1}$ is a shedding vertex.
\end{lemma}

\begin{proof}
    Suppose that $a_{n+1}$ is not a shedding vertex of $\kcut \lno \twoggp \rno$. Then there exists a facet $\sigma \in \del_{\kcut \lno \twoggp \rno}\lno a_{n+1} \rno$ which is not a facet of $\kcut \lno \twoggp \rno$, i.e., there exists 
\[
    \tau \in \kcut \lno \twoggp \rno \quad\text{with}\quad \sigma \subsetneq \tau.
\]
If $a_{n+1} \notin \tau$, then $\tau \in \del_{\kcut \lno \twoggp \rno}\lno a_{n+1} \rno$ and $\sigma \subsetneq \tau$, contradicting the fact that $\sigma$ is a facet. If $a_{n+1} \in \tau$, consider the following cases:

\begin{enumerate}[label=\arabic*.]
    \item If $a_{n} \in \tau$, set $\sigma' = \lno \tau \setminus \{a_{n+1}\} \rno \cup \lcu x \rcu,$     where $x$ is some element of $\tau^c$. Then $\sigma' \in \del_{\kcut \lno \twoggp \rno}\lno a_{n+1} \rno$ and $\sigma \subsetneq \sigma'$.

    \item If $a_{n} \notin \tau$, set $\sigma' = \lno \tau \setminus \lcu a_{n+1} \rcu \rno \cup \lcu a_n \rcu.$ Then $\sigma' \in \del_{\kcut \lno \twoggp \rno}\lno a_{n+1} \rno$ and $\sigma \subsetneq \sigma'$.
\end{enumerate}

In both cases, we obtain a contradiction to the assumption that $\sigma$ is a facet. Therefore $a_{n+1}$ is a shedding vertex of $\kcut \lno \twoggp \rno$.
\end{proof}

\begin{lemma}\label{lemma: b_n+1 is shedding vertex}
    For $\kcut \lno \twoggn \rno$, the vertex $b_{n+1}$ is a shedding vertex.
\end{lemma}

\begin{proof}
    Suppose $b_{n+1}$ is not a shedding vertex of $\kcut \lno \twoggn \rno$. Then there exists a facet $\sigma$ of $\del_{\kcut \lno \twoggn \rno}\lno b_{n+1} \rno$ which is not a facet of $\kcut \lno \twoggn \rno$. Hence, there exists $\tau \in \kcut \lno \twoggn \rno$ with $\sigma \subsetneq \tau.$

    If $b_{n+1} \notin \tau$, then $\tau \in \del_{\kcut \lno \twoggn \rno}\lno b_{n+1} \rno$ and $\sigma \subsetneq \tau$, contradicting the fact that $\sigma$ is a facet. Assume now that $b_{n+1} \in \tau$. Consider the following cases:

    \begin{enumerate}[label=\arabic*.]
        \item If $a_{n+1}, b_n \in \tau$, set $\tau' = \lno \tau \setminus \{b_{n+1}\} \rno \cup \lcu x \rcu$ for some element $x \in \tau^c$.

        \item If only $a_{n+1} \in \tau$, set $\tau' = \lno \tau \setminus \lcu b_{n+1} \rcu \rno \cup \lcu b_n \rcu.$

        \item If only $b_n \in \tau$, set $\tau' = \lno \tau \setminus \lcu b_{n+1} \rcu \rno \cup \lcu a_{n+1} \rcu.$

        \item Suppose $a_{n+1}, b_n \notin \tau$ and $\tau = \lcu x_1, x_2, \dots, x_{2n+2-k} \rcu$ with $x_1 < x_2 < \dots < x_{2n+2-k}$. Clearly $x_{2n+2-k} = b_{n+1}$.  
        If $x_{2n+1-k} = a_i$ for some $i \leq n$, set $\tau' = \lno \tau \setminus \lcu b_{n+1} \rcu \rno \cup \lcu b_i \rcu.$
        If $x_{2n+1-k} = b_i$ for some $i < n$, set $\tau' = \lno \tau \setminus \lcu b_{n+1} \rcu \rno \cup \lcu a_{i+1} \rcu.$
    \end{enumerate}

    In all four cases, $\tau' \in \del_{\kcut \lno \twoggn \rno}\lno b_{n+1} \rno$ and $\sigma \subsetneq \tau'$, contradicting the assumption that $\sigma$ is a facet of the deletion. Therefore $b_{n+1}$ is a shedding vertex of $\kcut \lno \twoggn \rno$.
\end{proof}

	\begin{lemma}\label{lemma: b_n is shedding vertex}
		For $\del_{\kcut\lno \twoggp \rno}\lno a_{n+1} \rno$, the vertex $b_n$ is a shedding vertex.
	\end{lemma}
	
	\begin{proof}
		For simplicity, denote $\Delta' = \del_{\kcut\lno \twoggp \rno}\lno a_{n+1} \rno$. 
		Suppose $b_n$ is not a shedding vertex of $\Delta'$. Then there exists a facet $\sigma \in \del_{\Delta'}\lno b_n \rno$ that is not a facet of $\Delta'$, i.e., there exists $\tau \in \Delta'$ such that 
		$\sigma \subsetneq \tau$.
		
		If $b_n \notin \tau$, then $\tau \in \del_{\Delta'}\lno b_n \rno$ and 
		$\sigma \subsetneq \tau$, contradicting that $\sigma$ is a facet. 
		Assume now that $b_n \in \tau$. Consider the following cases:
		
		\begin{enumerate}[label=\arabic*.]
			\item If $a_n \in \tau$, set $\tau' = \lno \tau \setminus \{b_n\} \rno \cup \lcu x \rcu,$ where $x$ is some element of $\tau^c$. Then $\tau' \in \Delta'$ and $\sigma \subsetneq \tau'$.
			
			\item If $a_n \notin \tau$, set $\tau' = \lno \tau \setminus \lcu b_n \rcu \rno \cup \lcu a_n \rcu.$ 
			Then $\tau' \in \Delta'$ and $\sigma \subsetneq \tau'$.
		\end{enumerate}
		
		In both cases, $\tau' \in \del_{\Delta'}\lno b_n \rno$ and $\sigma \subsetneq \tau'$, again contradicting the assumption that $\sigma$ is a facet. Therefore, $b_n$ is a shedding vertex of $\Delta'$.
	\end{proof}

	\begin{lemma}\label{lemma: a_n+1 is shedding vertex of del(b_n+1)}
		For $\del_{\kcut\lno \twoggn \rno}\lno b_{n+1} \rno$, the vertex $a_{n+1}$ is a shedding vertex.
	\end{lemma}
	
	\begin{proof}
		For simplicity, denote $\Delta'' = \del_{\kcut\lno \twoggn \rno}\lno b_{n+1} \rno$. 
		Suppose $a_{n+1}$ is not a shedding vertex of $\Delta''$. Then there exists a facet $\sigma \in \del_{\Delta''}\lno a_{n+1} \rno$ that is not a facet of $\Delta''$, i.e., there exists $\tau \in \Delta''$ with $\sigma \subsetneq \tau$. If $a_{n+1} \notin \tau$, then $\tau \in \del_{\Delta''}\lno a_{n+1} \rno$ and $\sigma \subsetneq \tau$, contradicting that $\sigma$ is a facet.
		
		Consider now the case when $a_{n+1} \in \tau$. Suppose $\tau = \{x_1, x_2, \ldots, x_{2n+1-k}, a_{n+1}\}$ with $x_1 < x_2 < \dots < x_{2n+1-k} < a_{n+1}$. The analysis is separated into the following cases:
		
		\begin{enumerate}[label=\arabic*.]
			\item If $x_{2n+1-k} = a_n$ and $b_n \notin \tau$, set $\tau' = \lno \tau \setminus \lcu a_{n+1} \rcu \rno \cup \lcu b_n \rcu.$ 
						
			\item If $x_{2n+1-k} = b_n$ and $a_n \notin \tau$, set $\tau' = \lno \tau \setminus \lcu a_{n+1} \rcu \rno \cup \lcu a_n \rcu.$ 
					
			\item If $x_{2n+1-k} = b_n$ and $a_n \in \tau$, choose any 
			$x \in \tau^c$ with $x \neq b_{n+1}$ and set $\tau' = \lno \tau \setminus \lcu a_{n+1} \rcu \rno \cup \lcu x \rcu.$
			
			\item If $x_{2n+1-k} \neq a_n, b_n$, proceed as follows. 
			For $i \geq 1$:  
			if $x_{2n+1-k} = b_{n-i}$, take $x = a_{n-i+1}$; 
			if $x_{2n+1-k} = a_{n-i}$, take $x = b_{n-i+1}$. 
			Then set $\tau' = \lno \tau \setminus \lcu a_{n+1} \rcu \rno \cup \lcu x \rcu.$
		\end{enumerate}
		
		In every case above, $\tau' \in \del_{\Delta''}\lno a_{n+1} \rno$ and 
		$\sigma \subsetneq \tau'$, which is a contradiction because $\sigma$ is a facet of $\del_{\Delta''}\lno a_{n+1} \rno$. 
		Therefore $a_{n+1}$ is a shedding vertex of $\Delta''$.
	\end{proof}
	
	The following statement is an immediate corollary of \Cref{lemma:link for cut}.
	
	\begin{corollary}\label{lemma: link of a_n+1 in 2n'-gg in k-cut} $\lk_{\kcut\lno \twoggn\rno}\lno b_{n+1} \rno = \kcut\lno \twoggp \rno,
		\qquad
		\lk_{\kcut\lno \twoggp \rno}\lno a_{n+1} \rno = \kcut\lno \twogg \rno.$
	\end{corollary}
	
	We now turn to the shellability of the complex $\kcut\lno \twogg \rno$ for all $3 \leq k \leq 2n-2$. It is shown in \cite[Proposition 5.1]{bayer2024cutB} that the complex $\Delta_{2n-2}\lno \twogg \rno$ is shellable only when $n \geq 3$. Consequently, we focus on the range $3 \leq k \leq 2n-3$.

	\begin{theorem}\label{lemma: Shellability of 2n'gg}
		For $3 \leq k \leq 2n-3$, the complex $\kcut \lno \twogg \rno$ is shellable.
	\end{theorem}
	
	\begin{proof}
		We prove the result by induction on $n$. For $k=3$, the claim follows from \cite[Proposition~5.6]{bayer2024cut}. Hence, assume $k \geq 4$. In this range, it is straightforward to verify that $\Delta_k(\mathcal{G}_{2\times 3})$ is shellable, so the base case of the induction holds.
		
		Assume that the result is true for some $n$ and consider $n+1$. Our goal is to show that $\kcut \lno \twoggn \rno$ is shellable. By \Cref{lemma: b_n+1 is shedding vertex}, the vertex $b_{n+1}$ is a shedding vertex of $\kcut \lno \twoggn \rno$. By \Cref{def:shellable}, it suffices to show that both $\lk_{\kcut \lno \twoggn\rno}\lno b_{n+1} \rno$ and $\del_{\kcut \lno \twoggn\rno}\lno b_{n+1} \rno$ are shellable. For convenience, denote these complexes by $\K_1$ and $\K_2$ respectively.

	\begin{enumerate}
		\item \textbf{Shellability of $\K_1$:}
		
		From \Cref{lemma: link of a_n+1 in 2n'-gg in k-cut} we obtain $\K_1 = \kcut\lno \twoggp \rno$. It therefore suffices to show that $\kcut\lno \twoggp \rno$ is shellable. By \Cref{lemma: 2n is shedding vertex of augumented grid graph}, the vertex $a_{n+1}$ is a shedding vertex of $\kcut\lno \twoggp \rno$. Hence, by \Cref{def:shellable}, it is enough to prove that both the complexes $\lk_{\kcut\lno \twoggp \rno}\lno a_{n+1} \rno$ and $ \del_{\kcut \lno \twoggp \rno}(a_{n+1})$ are shellable.
		
		By \Cref{lemma: link of a_n+1 in 2n'-gg in k-cut}, we have $\lk_{\kcut\lno \twoggp \rno}\lno a_{n+1} \rno = \kcut\lno \twogg \rno,$ which is shellable by the induction hypothesis. It remains to show that $\del_{\kcut \lno \twoggp \rno}(a_{n+1})$ is shellable. From \Cref{lemma: b_n is shedding vertex}, the vertex $b_n$ is a shedding vertex of $\del_{\kcut\lno\twoggp\rno}\lno a_{n+1} \rno$. Consequently, by \Cref{def:shellable}, it is enough to verify that both the link and the deletion of $b_n$ in $\del_{\kcut\lno\twoggp\rno}\lno a_{n+1} \rno$ are shellable.
		
		\begin{itemize}
			\itemsep1em
			\item \textbf{Shellability of the link of $b_n$ in $\del_{\kcut\lno\twoggp\rno}\lno a_{n+1} \rno$:}
			
			Let $\Gamma$ denote the link of $b_n$ in $\del_{\kcut\lno\twoggp\rno}\lno a_{n+1} \rno$. We order the facets of $\Gamma$ as follows. First, list all facets that contain $a_n$ in lexicographic order, followed by the remaining facets in arbitrary order. Concretely, let
			\[
			F_1, F_2, \dots, F_l, F_{l+1}, \dots, F_t
			\]
			be the facets of $\Gamma$, where $a_n \in F_i$ for $1 \le i \le l$ and $a_n \notin F_i$ for $l+1 \le i \le t$, with $F_1,\dots,F_l$ arranged lexicographically. We claim that this is a shelling order.
			
			For any facet $F$ among $F_1,\dots,F_l$, we have $F^c = C \sqcup \lcu a_{n+1}, b_n \rcu,$	where $C$ is any subset of $V\lno \mathcal{G}_{2 \times \lno n-1 \rno}\rno$ of size $k-1$. The induced subgraph on $F^c$ is always disconnected, so every subset of size $2n-k-1$ of $V\lno \mathcal{G}_{2 \times \lno n-1 \rno}\rno$, together with $a_n$, forms a facet of $\Gamma$. This complex is the cone over the $\lno 2n-k-2 \rno$–skeleton of the $\lno 2n-3 \rno$–dimensional simplex with apex $a_n$, namely $	\cone_{a_n}\lno \lno \Delta^{2n-3}\rno^{\lno 2n-k-2 \rno} \rno,$ which is shellable by \Cref{join is shellable}. It is also immediate that $F_1, F_2,\dots, F_l$ forms a shelling order.
			
			For any $F_j$ with $l+1 \le j \le t$, the intersection $\displaystyle \lno \bigcup_{s=1}^{\,j-1} \langle F_s \rangle \rno \cap \langle F_j \rangle$ is pure of dimension $2n-k-2$, since every boundary facet of $F_j$ already appears among the facets $F_i$ with $1 \le i \le l$. In particular, for any $x \in F_j$, consider the facet $F_j' = \lno F_j \setminus \lcu x \rcu \rno \cup \lcu a_n \rcu.$ Note that $\lvert F_j \cap F_j' \rvert = \lvert F_j \setminus \lcu x \rcu \rvert$ and that $F_j' < F_j$ in the ordering defined above. This shows that the ordering is a shelling order. Hence, $\Gamma$ is shellable.
			
			\item \textbf{Shellability of the deletion of $b_n$ in $\del_{\kcut \lno\twoggp\rno}\lno a_{n+1} \rno$:}
			
			Let $\Gamma'$ denote the deletion of $b_n$ in $\del_{\kcut \lno\twoggp\rno}\lno a_{n+1} \rno$. If a facet $F$ of $\Gamma'$ contains $a_n$, then $F^c = D \sqcup \lcu a_{n+1}, b_n \rcu,$ where $D$ is any subset of $V\lno \mathcal{G}_{2 \times \lno n-1 \rno}\rno$ of size $k-2$. The induced subgraph on $F^c$ is always disconnected, since $a_n \in F$ while $a_{n+1}$ and $b_n$ lie in $F^c$. 
			
			We order the facets of $\Gamma'$ by listing first all facets that contain $a_n$ in lexicographic order, followed by the remaining facets in arbitrary order. Repeating the same argument used for $\Gamma$ above, we conclude that this ordering is a shelling order. Hence, $\Gamma'$ is shellable.
		\end{itemize}
			
	Hence, $\del_{\kcut\lno\twoggp\rno}\lno a_{n+1} \rno$ is shellable, and consequently $\K_1 = \lk_{\kcut \lno\twoggn\rno}\lno b_{n+1} \rno$ is shellable.

\item \textbf{Shellability of $\K_2$:}
			
			From \Cref{lemma: a_n+1 is shedding vertex of del(b_n+1)}, the vertex $a_{n+1}$ is a shedding vertex of $\K_2$. Hence it is enough to show that both $\lk_{\K_2}(a_{n+1})$ and $\del_{\K_2}\lno a_{n+1} \rno$ are shellable.
			
			\begin{itemize}
				\item \textbf{Shellability of $\lk_{\K_2}\lno a_{n+1} \rno$:}
			For any facet $F$ of $\lk_{\K_2}\lno a_{n+1} \rno$, the vertices $a_{n+1}$ and $b_{n+1}$ do not belong to $F$. Thus $F$ is a subset of $V \lno \twogg \rno$ of size $2n+1-k$. Write
			\[
			F = \lcu x_1, x_2, \dots, x_{2n+1-k} \rcu
			\quad\text{with}\quad
			x_1 < x_2 < \dots < x_{2n+1-k}.
			\]
			We order the facets of $\lk_{\K_2}\lno a_{n+1} \rno$ as follows:
			\begin{enumerate}[label = \roman*.]
				\item First list all facets with $x_{2n-k}=a_n$ and $x_{2n+1-k}=b_n$ in lexicographic order;
				\item then list all facets with $x_{2n-k}\neq a_n$ and $x_{2n+1-k}=b_n$ in any order;
				\item then list all facets with $x_{2n+1-k}=a_n$ in any order;
				\item next list all facets with $x_{2n-k}=a_{n-1}$ and $x_{2n+1-k}=b_{n-1}$ in lexicographic order;
				\item then list all facets with $x_{2n-k}\neq a_{n-1}$ and $x_{2n+1-k}=b_{n-1}$ in any order;
				\item then list all facets with $x_{2n+1-k}=a_{n-1}$ in any order, and continue in this manner until all facets have been ordered.
			\end{enumerate}
			
			Observe that for the facets containing both $a_n$ and $b_n$, one may choose any $2n-1-k$ elements from $V\lno \mathcal{G}_{2 \times \lno n-1 \rno}\rno$. These facets form the simplicial complex $\Delta_1 = \langle \{a_n,b_n\}\rangle * \lno \Delta^{2n-3}\rno^{\lno 2n-2-k \rno},$ which is shellable by \Cref{join is shellable}, since $\lno \Delta^{2n-3}\rno^{\lno 2n-2-k \rno}$ is shellable by \Cref{shellability of skeleton}. The lexicographic ordering of these facets is clearly a shelling order.
			
			Now consider any facet $F_j = \lcu x_1, x_2, \dots, x_{2n+1-k} \rcu$ from the ordering. We show that this facet satisfies the shelling condition.
			
			\begin{enumerate}
				\item If $x_{2n-k}=a_n$ and $x_{2n+1-k}=b_n$, then the claim follows immediately.
				
				\item Let $x_{2n-k}\neq a_i$ and $x_{2n+1-k}=b_i$ for some $i \leq n$. In this case, every boundary facet of $\langle F_j\rangle$ except $F_j \setminus \{b_i\}$ appears earlier in the ordering in a facet containing $\{a_i,b_i\}$. Hence, any arrangement of these facets produces a valid shelling order.
				
				\item Let $x_{2n+1-k}=a_i$ for some $i \leq n$. This situation is similar to the previous case, except that $F_j \setminus \{b_i\}$ may also appear earlier in the ordering. This does not affect the shelling condition, so the ordering defined above still gives a shelling order.
				
				\item Let $x_{2n-k}=a_{i-1}$ and $x_{2n+1-k}=b_{i-1}$ for some $i \leq n$. In this case, all boundary facets of $\langle F_j\rangle$ appear earlier in the ordering. In particular,
				
				\begin{enumerate}
					\item For $x=b_{i-1}$, consider $F'_j = \lno F_j \setminus \lcu x \rcu \rno \cup \lcu b_i \rcu.$ Clearly, $F'_j < F_j$.
					
					\item For $x=a_{i-1}$, consider $F'_j = (F_j \setminus \lcu x \rcu) \cup \lcu a_i \rcu.$ Again, $F'_j < F_j$.
					
					\item For $x \neq a_{i-1}, b_{i-1}$, consider $F'_j = \lno F_j \setminus \lcu x \rcu \rno \cup \lcu a_i \rcu$ or $	F'_j = (F_j \setminus \lcu x \rcu) \cup \lcu b_i \rcu.$	In both cases, $F'_j < F_j$.
				\end{enumerate}
				Hence, $\lk_{\K_2}\lno a_{n+1} \rno$ is shellable. 
				
			\end{enumerate}
			
			\item \textbf{Shellability of $\del_{\K_2}\lno a_{n+1} \rno$:}
			For any facet $F$ of $\del_{\K_2}\lno a_{n+1} \rno$, the vertices $a_{n+1}$ and $b_{n+1}$ do not belong to $F$. Hence $F$ is a subset of $V\lno \twogg \rno$ of size $2n+2-k$. Write
\[ F = \lcu x_1, x_2, \dots, x_{2n+2-k} \rcu
\quad\text{with}\quad
x_1 < x_2 < \dots < x_{2n+1-k} < x_{2n+2-k}.
\]
We order the facets of $\del_{\K_2}\lno a_{n+1} \rno$ as follows:
\begin{enumerate}[label = \roman*.]
\item First list all facets with $x_{2n+1-k}=a_n$ and $x_{2n+2-k}=b_n$ in lexicographic order;
\item then list all facets with $x_{2n+1-k}\neq a_n$ and $x_{2n+2-k}=b_n$ in any order;
\item then list all facets with $x_{2n+2-k}=a_n$ in any order;
\item next list all facets with $x_{2n+1-k}=a_{n-1}$ and $x_{2n+2-k}=b_{n-1}$ in lexicographic order;
\item then list all facets with $x_{2n+1-k}\neq a_{n-1}$ and $x_{2n+2-k}=b_{n-1}$ in any order;
\item then list all facets with $x_{2n+2-k}=a_{n-1}$ in any order, and continue in this manner until all facets have been listed.
\end{enumerate}

Using the same arguments as in the case of $\lk_{\K_2}\lno a_{n+1} \rno$, we conclude that $\del_{\K_2}\lno a_{n+1} \rno$ is also shellable.	
		\end{itemize}
		\end{enumerate}
		
		This completes the proof of \Cref{lemma: Shellability of 2n'gg}.
	\end{proof}

\section*{Acknowledgements} We thank the anonymous referees for their constructive comments and suggestions, which significantly improved the paper. We are also grateful to Sheila Sundaram for her insightful and constructive feedback on the earlier draft of this manuscript. We also extend our thanks to Krishna Menon for engaging in numerous discussions on relevant topics. Himanshu Chandrakar gratefully acknowledges the assistance provided by the Council of Scientific and Industrial Research (CSIR), India, through grant 09/1237(15675)/2022-EMR-I. Anurag Singh is partially supported by the Start-up Research Grant SRG/2022/000314 from SERB, DST, India.

\bibliographystyle{abbrv}

\begin{thebibliography}{10}
	
	\bibitem{Henryvietoris}
	M.~Adamaszek and H.~Adams.
	\newblock On {V}ietoris-{R}ips complexes of hypercube graphs.
	\newblock {\em J. Appl. Comput. Topol.}, 6(2):177--192, 2022.
	
	\bibitem{Aharoniindependent}
	R.~Aharoni, E.~Berger, and R.~Ziv.
	\newblock Independent systems of representatives in weighted graphs.
	\newblock {\em Combinatorica}, 27(3):253--267, 2007.
	
	\bibitem{Alonchromatic}
	N.~Alon, P.~Frankl, and L.~Lov\'asz.
	\newblock The chromatic number of {K}neser hypergraphs.
	\newblock {\em Trans. Amer. Math. Soc.}, 298(1):359--370, 1986.
	
	\bibitem{BK06}
	E.~Babson and D.~N. Kozlov.
	\newblock Complexes of graph homomorphisms.
	\newblock {\em Israel J. Math.}, 152:285--312, 2006.
	
	\bibitem{Barmakstarclusters}
	J.~A. Barmak.
	\newblock Star clusters in independence complexes of graphs.
	\newblock {\em Adv. Math.}, 241:33--57, 2013.
	
	\bibitem{bayer2024cut}
	M.~Bayer, M.~Denker, M.~Jeli\'c{}~Milutinovi\'c, R.~Rowlands, S.~Sundaram, and
	L.~Xue.
	\newblock Topology of cut complexes of graphs.
	\newblock {\em SIAM J. Discrete Math.}, 38(2):1630--1675, 2024.
	
	\bibitem{bayer2024total}
	M.~Bayer, M.~Denker, M.~Jeli\'c{}~Milutinovi\'c, R.~Rowlands, S.~Sundaram, and
	L.~Xue.
	\newblock Total cut complexes of graphs.
	\newblock {\em Discrete Comput. Geom.}, 2024.
	
	\bibitem{bayer2024cutB}
	M.~Bayer, M.~Denker, M.~Jeli\'c{}~Milutinovi\'c, S.~Sundaram, and L.~Xue.
	\newblock Topology of cut complexes {II}.
	\newblock {\em SIAM J. Discrete Math.}, 39(2), 2025.
	
	\bibitem{Bjornercombinatorial}
	A.~Bj\"orner.
	\newblock Some combinatorial and algebraic properties of {C}oxeter complexes
	and {T}its buildings.
	\newblock {\em Adv. in Math.}, 52(3):173--212, 1984.
	
	\bibitem{BjornerTopoMethods}
	A.~Bj\"orner.
	\newblock Topological methods.
	\newblock In {\em Handbook of combinatorics, {V}ol. 1, 2}, pages 1819--1872.
	Elsevier Sci. B. V., Amsterdam, 1995.
	
	\bibitem{Bjornershellable}
	A.~Bj\"orner and M.~L. Wachs.
	\newblock Shellable nonpure complexes and posets. {I}.
	\newblock {\em Trans. Amer. Math. Soc.}, 348(4):1299--1327, 1996.
	
	\bibitem{Wachs97b}
	A.~Bj\"orner and M.~L. Wachs.
	\newblock Shellable nonpure complexes and posets. {II}.
	\newblock {\em Trans. Amer. Math. Soc.}, 349(10):3945--3975, 1997.
	
	\bibitem{BH17}
	B.~Braun and W.~K. Hough.
	\newblock Matching and independence complexes related to small grids.
	\newblock {\em Electron. J. Combin.}, 24(4):Paper No. 4.18, 20, 2017.
	
	\bibitem{Bright04}
	G.~R. Brightwell and P.~Winkler.
	\newblock Graph homomorphisms and long range action.
	\newblock In {\em Graphs, morphisms and statistical physics}, volume~63 of {\em
		DIMACS Ser. Discrete Math. Theoret. Comput. Sci.}, pages 29--47. Amer. Math.
	Soc., Providence, RI, 2004.
	
	\bibitem{camara2016topological}
	P.~G. Camara, D.~I. Rosenbloom, K.~J. Emmett, A.~J. Levine, and R.~Rabadan.
	\newblock Topological data analysis generates high-resolution, genome-wide maps
	of human recombination.
	\newblock {\em Cell systems}, 3(1):83--94, 2016.
	
	\bibitem{chanviralevolution}
	J.~M. Chan, G.~Carlsson, and R.~Rabadan.
	\newblock Topology of viral evolution.
	\newblock {\em Proc. Natl. Acad. Sci. USA}, 110(46):18566--18571, 2013.
	
	\bibitem{chandrakar2024perfect}
	H.~Chandrakar and A.~Singh.
	\newblock Perfect matching complexes of polygonal line tiling.
	\newblock {\em Ann. Comb.}, published online, 2025.
	
	\bibitem{Chauhan2025}
	P.~Chauhan, S.~Shukla, and K.~Vinayak.
	\newblock Shellability of 3-cut complexes of squared cycle graphs.
	\newblock {\em Journal of Homotopy and Related Structures}, 20(1):163--193,
	2025.
	
	\bibitem{deshpandedomination}
	P.~Deshpande, S.~Shukla, and A.~Singh.
	\newblock Distance {$r$}-domination number and {$r$}-independence complexes of
	graphs.
	\newblock {\em European J. Combin.}, 102:Paper No. 103508, 14, 2022.
	
	\bibitem{deshpande2020higher}
	P.~Deshpande and A.~Singh.
	\newblock Higher independence complexes of graphs and their homotopy types.
	\newblock {\em J. Ramanujan Math. Soc.}, 36(1):53--71, 2021.
	
	\bibitem{Dochter09}
	A.~Dochtermann and A.~Engstr\"om.
	\newblock Algebraic properties of edge ideals via combinatorial topology.
	\newblock {\em Electron. J. Combin.}, 16(2):Research Paper 2, 24, 2009.
	
	\bibitem{dochter12}
	A.~Dochtermann and A.~Engstr\"om.
	\newblock Cellular resolutions of cointerval ideals.
	\newblock {\em Math. Z.}, 270(1-2):145--163, 2012.
	
	\bibitem{Dochtermannwarmath}
	A.~Dochtermann and R.~Freij-Hollanti.
	\newblock Warmth and edge spaces of graphs.
	\newblock {\em Adv. in Appl. Math.}, 96:176--194, 2018.
	
	\bibitem{DS12}
	A.~Dochtermann and C.~Schultz.
	\newblock Topology of {H}om complexes and test graphs for bounding chromatic
	number.
	\newblock {\em Israel J. Math.}, 187:371--417, 2012.
	
	\bibitem{ER98}
	J.~A. Eagon and V.~Reiner.
	\newblock Resolutions of {S}tanley-{R}eisner rings and {A}lexander duality.
	\newblock {\em J. Pure Appl. Algebra}, 130(3):265--275, 1998.
	
	\bibitem{forman1998morse}
	R.~Forman.
	\newblock Morse theory for cell complexes.
	\newblock {\em Adv. Math.}, 134(1):90--145, 1998.
	
	\bibitem{forman2002user}
	R.~Forman.
	\newblock A user's guide to discrete {M}orse theory.
	\newblock {\em S\'em. Lothar. Combin.}, 48:Art. B48c, 35, 2002.
	
	\bibitem{froberg1990}
	R.~Fr\"oberg.
	\newblock On {S}tanley-{R}eisner rings.
	\newblock In {\em Topics in algebra, {P}art 2 ({W}arsaw, 1988)}, volume 26,
	Part 2 of {\em Banach Center Publ.}, pages 57--70. PWN, Warsaw, 1990.
	
	\bibitem{graham94}
	R.~L. Graham, D.~E. Knuth, and O.~Patashnik.
	\newblock {\em Concrete Mathematics: A Foundation for Computer Science}.
	\newblock Addison-Wesley, Reading, MA, second edition, 1994.
	
	\bibitem{Hatcherat}
	A.~Hatcher.
	\newblock {\em Algebraic topology}.
	\newblock Cambridge University Press, Cambridge, 2002.
	
	\bibitem{hibi92}
	T.~Hibi.
	\newblock Combinatorics of simplicial complexes and complex polyhedra.
	\newblock {\em S\=ugaku}, 44(2):147--160, 1992.
	
	\bibitem{hoch16}
	M.~Hochster.
	\newblock Cohen-{M}acaulay varieties, geometric complexes, and combinatorics.
	\newblock In {\em The mathematical legacy of {R}ichard {P}. {S}tanley}, pages
	219--229. Amer. Math. Soc., Providence, RI, 2016.
	
	\bibitem{hockey}
	C.~H. Jones.
	\newblock Generalized hockey stick identities and {$N$}-dimensional
	blockwalking.
	\newblock {\em Fibonacci Quart.}, 34(3):280--288, 1996.
	
	\bibitem{jonsson2008simplicial}
	J.~Jonsson.
	\newblock {\em Simplicial complexes of graphs}, volume 1928.
	\newblock Springer, 2008.
	
	\bibitem{kozlovCombAlgTopo}
	D.~Kozlov.
	\newblock {\em Combinatorial algebraic topology}, volume~21 of {\em Algorithms
		and Computation in Mathematics}.
	\newblock Springer, Berlin, 2008.
	
	\bibitem{lovasz1978}
	L.~Lov\'asz.
	\newblock Kneser's conjecture, chromatic number, and homotopy.
	\newblock {\em J. Combin. Theory Ser. A}, 25(3):319--324, 1978.
	
	\bibitem{Matsushitamatching}
	T.~Matsushita.
	\newblock Matching complexes of small grids.
	\newblock {\em Electron. J. Combin.}, 26(3):Paper No. 3.1, 8, 2019.
	
	\bibitem{Matsushitantimes4}
	T.~Matsushita and S.~Wakatsuki.
	\newblock Independence complexes of {$(n \times 4)$} and {$(n \times 5)$}-grid
	graphs.
	\newblock {\em Topology Appl.}, 334:Paper No. 108541, 18, 2023.
	
	\bibitem{matsushita2023dominance}
	T.~Matsushita and S.~Wakatsuki.
	\newblock Dominance complexes, neighborhood complexes and combinatorial
	alexander duals.
	\newblock {\em J. Combin. Theory Ser. A}, 211, 2025.
	
	\bibitem{Meshulamdomination}
	R.~Meshulam.
	\newblock Domination numbers and homology.
	\newblock {\em J. Combin. Theory Ser. A}, 102(2):321--330, 2003.
	
	\bibitem{Singh20}
	A.~Singh.
	\newblock Bounded degree complexes of forests.
	\newblock {\em Discrete Math.}, 343(10):112009, 7, 2020.
	
	\bibitem{Singhhighermatching}
	A.~Singh.
	\newblock Higher matching complexes of complete graphs and complete bipartite
	graphs.
	\newblock {\em Discrete Math.}, 345(4):Paper No. 112761, 10, 2022.
	
	\bibitem{Stan96}
	R.~P. Stanley.
	\newblock {\em Combinatorics and commutative algebra}, volume~41 of {\em
		Progress in Mathematics}.
	\newblock Birkh\"auser Boston, Inc., Boston, MA, second edition, 1996.
	
	\bibitem{StanleyenumerativeI}
	R.~P. Stanley.
	\newblock {\em Enumerative combinatorics. {V}ol. 1}, volume~49 of {\em
		Cambridge Studies in Advanced Mathematics}.
	\newblock Cambridge University Press, Cambridge, 1997.
	\newblock With a foreword by Gian-Carlo Rota, Corrected reprint of the 1986
	original.
	
	\bibitem{van13}
	A.~Van~Tuyl.
	\newblock A beginner's guide to edge and cover ideals.
	\newblock In {\em Monomial ideals, computations and applications}, volume 2083
	of {\em Lecture Notes in Math.}, pages 63--94. Springer, Heidelberg, 2013.
	
	\bibitem{Wachschessboard}
	M.~L. Wachs.
	\newblock Topology of matching, chessboard, and general bounded degree graph
	complexes.
	\newblock volume~49, pages 345--385. 2003.
	\newblock Dedicated to the memory of Gian-Carlo Rota.
	
\end{thebibliography}

%\nocite{*}
{\Addresses}

\end{document}